\documentclass{article}
\pdfoutput=1
\usepackage[pdftex,colorlinks=true]{hyperref}
\usepackage{amsmath,amssymb,mathrsfs,amsthm}
\usepackage{array,enumerate}
\usepackage{latexsym}
\usepackage{mathrsfs}
\usepackage{amscd}
\usepackage{url}
\usepackage{comment}
\usepackage{framed, color}
\usepackage{ulem}
\usepackage{braket}
\usepackage[all]{xy}
\usepackage{geometry}
\usepackage{stmaryrd}

\theoremstyle{plain}
\newtheorem{thm}{Theorem}[section]
\newtheorem*{thm*}{Theorem}
\makeatletter
    \let\c@equation\c@thm
    
    \@addtoreset{equation}{section}
\makeatother
\newcommand{\bvec}[1]{\mbox{\boldmath $#1$}}
\newtheorem{cor}[thm]{Corollary}
\newtheorem*{cor*}{Corollary}
\newtheorem{prop}[thm]{Proposition}
\newtheorem*{prop*}{Proposition}
\newtheorem{lem}[thm]{Lemma}
\newtheorem*{lem*}{Lemma}

\newtheorem*{sublem*}{Lemma}
\newtheorem{claim}[thm]{Claim}
\newtheorem*{claim*}{Claim}
\theoremstyle{definition}
\newtheorem{dfn}[thm]{Definition}
\newtheorem*{dfn*}{Definition}

\newtheorem*{conj*}{Conjecture}

\newtheorem*{cnst*}{Construction}

\newtheorem*{pers*}{Perspective}

\newtheorem*{dig*}{Digression}
\theoremstyle{remark}
\newtheorem{rem}[thm]{Remark}
\newtheorem*{rem*}{Remark}

\newtheorem*{ex*}{Example}

\newtheorem*{nt*}{Notation}

\newcommand{\void}[1]{}

\title{Cube invariance of higher Chow groups with modulus}
\author{Hiroyasu Miyazaki
	\thanks{
	RIKEN iTHES, and 
	Graduate School of Mathematical Sciences, the University of Tokyo.
	\newline
	\textit{E-mail address 1}: \protect\url{hiroyasu.miyazaki@riken.jp}
	\textit{E-mail address 2}: \protect\url{miyazaki@ms.u-tokyo.ac.jp}} }
	
\begin{document}

\title{Cube invariance of higher Chow groups with modulus
}


\author{Hiroyasu Miyazaki
	\thanks{
	RIKEN iTHES, and 
	Graduate School of Mathematical Sciences, the University of Tokyo.
	\newline
	\textit{E-mail address 1}: \protect\url{hiroyasu.miyazaki@riken.jp}
	\textit{E-mail address 2}: \protect\url{miyazaki@ms.u-tokyo.ac.jp}} }

\maketitle

\begin{abstract}
The higher Chow group with modulus was introduced by Binda-Saito as a common generalization of Bloch's higher Chow group and the additive higher Chow group.
In this paper, we study invariance properties of the higher Chow group with modulus.
First, we formulate and prove ``cube invariance," which generalizes $\mathbb{A}^1$-homotopy invariance of Bloch's higher Chow group.
Next, we introduce the nilpotent higher Chow group with modulus, as an analogue of the nilpotent algebraic $K$-group, and define a module structure on it over the big Witt ring of the base field.
We deduce from the module structure that the higher Chow group with modulus with appropriate coefficients satisfies $\mathbb{A}^1$-homotopy invariance.
We also prove that $\mathbb{A}^1$-homotopy invariance implies independence from the multiplicity of the modulus divisors.
\end{abstract}

\subsection*{Acknowledgements}
The results in the present paper are based on the author's Ph.D. thesis work.
The author would like to express his deepest gratitude to Professor Shuji Saito for invitation to the truly interesting theory of modulus and for strong and continuous encouragements in the author's adversity.
My heartfelt appreciation also goes to Professor Tomohide Terasoma for his kind and  helpful advices, which improved the arguments in the paper.
The author is grateful to Wataru Kai and Ryomei Iwasa for the friendship and exciting discussions, which led the author to a better understanding of the theory.
The author would particularly like to thank Wataru Kai for suggesting key points of the proof of Proposition \ref{vanishing}.
The formulation of cube invariance came to the author's mind when he was in Niseko, Hokkaido, Japan, to attend the workshop ``Regulators in Niseko" in September, 2015.
The author thanks the organizers of the workshop.
The work in the paper was supported by a Grant-in-Aid for JSPS Fellows (Grant Number 15J08833) and the Program for Leading Graduate Schools, MEXT, Japan.
It was partially supported by RIKEN iTHES project, too.

\tableofcontents

\section{Introduction}

For a separated scheme $X$ which is equidimensional of finite type over a field $k$, S. Bloch introduced in \cite{B} the higher Chow group $\mathrm{CH}^r (X,\ast)$ of codimension $r$ to give a description of Quillen's algebraic $K$-theory by using algebraic cycles.
It is a very nice cohomology theory, which is isomorphic to the motivic cohomology group defined by V. Voevodsky.
However, there is a problem that the higher Chow group forgets the information of {\it non-reduced} schemes.
Indeed, though $K_\ast (X) \neq K_\ast (X_{\mathrm{red}})$ in general, we have $\mathrm{CH}^r (X,\ast) = \mathrm{CH}^r (X_{\mathrm{red }},\ast)$.

\ 

Recently, F. Binda and S. Saito defined in \cite{BS} the higher Chow group with modulus, as a generalization of Bloch's higher Chow group, to capture the information of non-reduced schemes.
It is defined for a pair $(X,D)$ and is denoted by $\mathrm{CH}^r (X|D,\ast )$, where $X$ is a scheme and $D$ is an effective Cartier divisor on $X$.
The most important point is that the divisor $D$ has the information of multiplicity, which means that $D$ is not necessarily reduced.
It is expected that $\mathrm{CH}^r (X|D,\ast )$ is a nice counterpart of the relative $K$-group $K_\ast (X,D)$.
For example, R. Iwasa and W. Kai constructed in \cite{IK} the relative Chern class, which relates the two groups.
We would like to note that the higher Chow group with modulus is also a generalization of the additive higher Chow group, which was introduced by S. Bloch and H. Esnault in \cite{BE}, \cite{BE2} and studied by A. Krishna, M. Levine, J. Park and K. R\"{u}lling (e.g. \cite{KL}, \cite{KP}, \cite{Ru}).

\ 

The main purpose of the present paper is to study invariance properties of the higher Chow group with modulus.
In \cite[Theorem 2.1]{B}, it is shown that the higher Chow group satisfies $\mathbb{A}^1$-homotopy invariance
\[
\mathrm{CH}^r (X,\ast) \stackrel{\cong}{\longrightarrow} \mathrm{CH}^r (X \times \mathbb{A}^1 , \ast),
\]
which is the most fundamental property.
However, the higher Chow group with modulus does not satisfy $\mathbb{A}^1$-homotopy invariance, as well as the relative $K$-group does not.
Our aim is to answer the following questions:
\begin{itemize}
\item[Q1.]
What is the generalization of $\mathbb{A}^1$-homotopy invariance for the higher Chow group with modulus?
\item[Q2.]
What is the ``error term" of $\mathbb{A}^1$-homotopy invariance, and what is its structure like?
\end{itemize}

Moreover, we also consider the following question, which turns out to be closely related to invariance properties:
\begin{itemize}
\item[Q3.]
How does the higher Chow group with modulus depend on the multiplicity of the divisor $D$?
\end{itemize}

Our first result is the following theorem, which answers Q1.
To formulate the statement, we need to generalize the definition of the higher Chow group with modulus in \S 2.
Though the higher Chow group with modulus was originally defined for a pair $(X,D)$ with $D$ an effective Cartier divisor on $X$, we permit the divisor $D$ to be {\it non-effective}.
The definition coincides with the original one when $D$ is effective.
An advantage of this generalization is that we can consider the following special pair
\[
\overline{\square}^{(-1)} = (\mathbb{P}^1, -\{\infty \}),
\]
which we call the {\it minus cube}.
Moreover, for two pairs $\mathscr{X}=(X,D)$ and $\mathscr{Y}=(Y,E)$, we define the tensor product 
\[
\mathscr{X} \otimes \mathscr{Y} = (X \times Y ,D \times Y + X \times E).
\]

\begin{thm}\label{main-1} (Cube invariance. See Theorem \ref{cube-invariance})
For any pair $\mathscr{X}=(X,D)$ and for any non-negative integer $r \geq 0$, we have a canonical isomorphism
\[
\mathrm{CH}^r (\mathscr{X} ,q) \xrightarrow{\cong } \mathrm{CH}^r (\mathscr{X} \otimes \overline{\square}^{(-1)} ,q).
\]
If $D$ is a trivial divisor, then the isomorphism coincides with $\mathbb{A}^1$-homotopy invariance.
\end{thm}

This result is related to the work of B. Kahn,  S. Saito and T. Yamazaki in \cite{KSY}.
They constructed the category of motives with modulus, which is an extension of Voevodsky's category of motives.
Their idea is to consider sheaves $\mathcal{F}$ on the category of pairs $\mathscr{X}=(X,D)$ satisfying the cube invariance $\mathcal{F}(\mathscr{X})\cong \mathcal{F}(\mathscr{X} \otimes \overline{\square})$, where $\overline{\square} = (\mathbb{P}^1 ,\{\infty \})$.
The difference of signs of the infinity point comes from the difference of functoriality of $\mathcal{F}$ and $\mathrm{CH}^r (\mathscr{X},\ast )$.
A sheaf $\mathcal{F}$ is supposed to be contravariantly functorial with respect to {\it admissible} morphisms, while the higher Chow group with modulus is contravariantly functorial with respect to {\it coadmissible} (and flat) morphisms (see Definition \ref{adm-coadm}).
To bridge these two types of functorialities, we will need to generalize Poincar\'{e} duality theorem in our setting.

The strategy of the proof of the above theorem is the same as that of $\mathbb{A}^1$-homotopy invariance of the higher Chow group in \cite{B} (see \S \ref{pf-of-inv}).
The key ingredients are Rigidity-lemma (Lemma \ref{rigidity}) and Moving-lemma (Theorem \ref{moving-lemma} and Corollary \ref{special-moving-lemma}).
\S \ref{section-cube-invariance} is devoted to the proofs of these results.

\ 

To answer Q2, we briefly recall C. Weibel's results on the nilpotent $K$-theory.
The nilpotent $K$-group $\mathrm{NK}_\ast (X)$ of a scheme $X$ is defined to be the cokernel of the map $\mathrm{pr}_1^\ast : K_\ast (X) \to K_\ast (X \times \mathbb{A}^1 )$ of Quillen's $K$-groups induced by the first projection.
Since the zero section induces a retraction of $\mathrm{pr}_1^\ast$, we obtain a canonical splitting $K_\ast (X \times \mathbb{A}^1 ) \cong K_\ast (X) \oplus \mathrm{NK}_\ast (X)$.
This means that Quillen's $K$-group satisfies $\mathbb{A}^1$-homotopy invariance if and only if the nilpotent $K$-group vanishes.
In \cite{Wei}, Weibel constructed a module structure on $\mathrm{NK}_\ast (X)$ over the big Witt ring $\mathbb{W}(k)$ of the base field $k$.
He deduced from the module structure the following vanishing results:
\begin{itemize}
\item If the characteristic of $k$ is $p>0$, then $\mathrm{NK}_\ast (X)$ is a $p$-group.
In particular, it vanishes if we ignore $p$.
\item If the characteristic of $k$ is zero, then $\mathrm{NK}_\ast (X)$ is a $k$-vector space. In particular, it vanishes in finite coefficients.
\end{itemize}

In \S 5, we introduce the {\it nilpotent higher Chow group with modulus}, denoted by $\mathrm{NCH}^r (\mathscr{X},\ast )$, which fits in the following canonical splitting:
\[
\mathrm{CH}^r (\mathscr{X} \otimes \mathbb{A}^1 ,\ast ) \cong \mathrm{CH}^r (\mathscr{X},\ast ) \oplus \mathrm{NCH}^r (\mathscr{X},\ast ).
\]
Here, we regard the scheme $\mathbb{A}^1$ as the pair $(\mathbb{A}^1 ,\emptyset )$. Therefore, by definition, $\mathscr{X} \otimes \mathbb{A}^1 =(X \times \mathbb{A}^1 ,D \times \mathbb{A}^1 )$.
The following theorem is an analogue of Weibel's result, which answers Q2.
In fact, for any abelian group $A$, we can define the higher Chow group with modulus $\mathrm{CH}^r (\mathscr{X},\ast ,A)$ and the nilpotent higher Chow group with modulus $\mathrm{NCH}^r (\mathscr{X},\ast ,A)$ {\it with coefficients in} $A$ (see Definition \ref{degenerate-cycle}).
For example, $\mathrm{NCH}^r (\mathscr{X},\ast ,\mathbb{Z}[1/p]) = \mathrm{NCH}^r (\mathscr{X},\ast ) \otimes_\mathbb{Z} \mathbb{Z}[1/p]$.
They also satisfy the same splitting as above.

\begin{thm}\label{main-vanishing} ($\mathbb{A}^1$-homotopy invariance. See Theorem \ref{Witt-action}, \ref{vanishing-of-NK} and Corollary \ref{hom-inv-part})
Let $A$ be an abelian group.
For any pair $\mathscr{X}=(X,D)$ and for any non-negative integer $r \geq 0$, there exists a continuous module strucutre on the nilpotent higher Chow group $\mathrm{NCH}^r (\mathscr{X},\ast ,A)$ over the big Witt ring $\mathbb{W}(k)$ of the base field $k$.
In particular, we have the following results:
\begin{itemize}
\item[$\mathrm{(a)}$] If the characteristic of $k$ is $p>0$, then $\mathrm{NCH}^r (\mathscr{X},\ast ,A)$ is a $p$-group.
In particular, it vanishes if $A$ is a $\mathbb{Z}[1/p]$-module (e.g. $A=\mathbb{Z}[1/p]$, or $A=\mathbb{Z}/l^n$ for $(l,p)=1$).

\item[$\mathrm{(b)}$] If the characteristic of $k$ is zero, then $\mathrm{NCH}^r (\mathscr{X},\ast ,A)$ is a $k$-vector space. In particular, it vanishes if $A$ is a torsion abelian group.
\end{itemize}
Moreover, if $\mathrm{NCH}^r (\mathscr{X},\ast ,A)$ vanishes, the higher Chow group with modulus with coefficients in $A$ satisfies $\mathbb{A}^1$-homotopy invariance i.e.
\[
\mathrm{CH}^r (\mathscr{X},\ast ,A) \xrightarrow{\cong } \mathrm{CH}^r (\mathscr{X}\otimes \mathbb{A}^1 ,\ast ,A).
\]
\end{thm}
In the positive characteristic case, we can give a more direct proof of $\mathbb{A}^1$-homotopy invariance using Frobenius argument (see Remark \ref{direct-proof}), without using the nilpotent higher Chow group with modulus.

\ 

Finally, we would like to answer Q3.
In the following, we assume that the divisors $D$ are all {\it effective}.
The main goal of \S \ref{section-indep} is the following result.

\begin{thm}\label{main-4}
(Independence. See Corollary \ref{comparison-reduced})
Consider pairs $(X,D)$ and $(X,D')$ such that the supports $|D|$ and $|D'|$ coincide.
Let $A$ be an abelian group, and assume one of the following conditions:
\begin{itemize}
\item
The characteristic of $k$ is $p>0$ and $A$ is a $\mathbb{Z}[1/p]$-module.
\item
The characteristic of $k$ is zero and $A$ is a torsion abelian group.
\end{itemize}
Then, for any non-negative integer $r$, we have a canonical isomorphism
\[
\mathrm{CH}^r (X|D,\ast, A) \cong \mathrm{CH}^r (X|D',\ast, A).
\]
In other words, the higher Chow group with modulus with coefficients in $A$ is independent of the multiplicity of the divisors $D$.
\end{thm}

This result is closely related to \cite[Theorem 2.7]{BCKS}, which states that if $X$ is proper over $k$, then the ``unipotent part" (=non-homotopcally invariant part) of
$\mathrm{CH}_0 (\mathscr{X})=\mathrm{CH}^{\dim X} (\mathscr{X},0)$ is $p$-primary torsion when $\mathrm{char}(k)=p>0$ and that it is divisible when $\mathrm{char}(k)=0$.
The group $\mathrm{CH}_0 (\mathscr{X})=\mathrm{CH}_0 (X|D)$ is called the Chow group of zero cycles with modulus, and it is used by M. Kerz and S. Saito in \cite{KeS} to give a cycle-theoretic description of the \'{e}tale fundamental group of the interior $X \setminus |D|$.

The strategy of the proof of the above theorem is to compare both sides of the isomorphism with an intermediate abelian group called the ``na\"{i}ve" higher Chow group with modulus, which is {\it a priori} independent of the multiplicity of the divisors.
We use $\mathbb{A}^1$-homotopy invariance (Theorem \ref{main-vanishing}) for the comparison.
We remark that we also obtain an independence result (Theorem \ref{comparison-reduced-coh}) for {\it relative motivic cohomology groups} $H_{\mathcal{M}}^\ast (\mathscr{X},\mathbb{Z}(r))$, introduced in \cite{BS} (see Definition \ref{relative-mot}), which is defined as the Nisnevich hypercohomology group of the sheaf of the cycle complexes with modulus.

\ 

Moreover, supposing that $D \leq D'$, we study the structure of the ``difference" between the groups $\mathrm{CH}^r (X|D,\ast)$ and $\mathrm{CH}^r (X|D',\ast)$.
Precisely speaking, we consider the ``cycle complexes with modulus" $z^r (X|D,\ast)$ and $z^r (X|D',\ast)$, whose homology groups are by definition the higher Chow groups with modulus $\mathrm{CH}^r (X|D,\ast)$ and $\mathrm{CH}^r (X|D',\ast)$, respectively.
We have a canonical inclusion of complexes $z^r (X|D',\ast) \subset z^r (X|D,\ast)$.
Then, we have the following result.

\begin{thm}(See Corollary \ref{difference-str})
Consider pairs $(X,D)$ and $(X,D')$ such that $|D|=|D'|$ and $D \leq D'$.
Then, for any integer $q$, the homology group $\mathrm{H}_q \left( \frac{z^r (X|D,\ast )}{z^r (X|D',\ast )} \right)$ of the quotient complex is a $p$-group if the characteristic of $k$ is $p>0$, and is a $\mathbb{Q}$-vector space if the characteristic of $k$ is zero.
\end{thm}

\section{The higher Chow group with modulus}

\subsection{Preliminaries on modulus pairs}

We recall the notion of modulus pairs.
In the introduction, a modulus pair is denoted by $\mathscr{X}=(X,D)$, but in the following, it is denoted by $\mathscr{X}=(\overline{X},X^\infty )$.


\begin{dfn}\label{modulus-pair}
A {\it modulus pair} is a pair $\mathscr{X} = (\overline{X},X^\infty )$ such that 
\begin{itemize}
\item[(1)] $\overline{X}$ is a separated, equidimensional scheme essentially of finite type over $k$,

\item[(2)] $X^\infty $ is a Cartier divisor on $\overline{X}$ (the case $|X^\infty |=\emptyset $ is allowed).

\end{itemize}
We call $\mathscr{X}^\circ :=\overline{X} - |X^\infty |$ the {\it interior} of $\mathscr{X}$.
We say that $\mathscr{X}$ is {\it effective} if $X^\infty $ is effective. 
\end{dfn}

The following lemmas are very important in the modulus theory, and will be used frequently in the rest of the paper.

\begin{lem}\label{effectivity-criterion}
Let $f : Y \to X$ be a surjective map of normal integral $k$-schemes.
Let $D$ be a (Cartier) divisor on $X$.
Then, the pullback $f^\ast (D)$ of $D$ is effective if and only if $D$ is effective.
\end{lem}
\begin{proof}
The ``if part'' is trivial.
The ``only if part'' is  \cite[Lemma 2.2]{KP}.
 \end{proof}

\begin{lem}\label{containment-lemma} (Containment lemma)
Let $\mathscr{X}=(\overline{X},X^\infty )$ be a modulus pair. Let $V \subset \mathscr{X}^\circ = \overline{X} - |X^\infty |$ be an integral closed subscheme, and let $V' \subset V$ be an integral closed subscheme of $V$. 
Let $\overline{V} \subset \overline{X}$ be the closure of $V$ in $X$, and let $\overline{V}^N$ be its normalization.
Denote by $\nu_V$ the composite map $\overline{V}^N \xrightarrow{} \overline{V} \xrightarrow{\mathrm{incl.}} \overline{X}$.
Define the map $\nu_{V'} : \overline{V'}^N \to  \overline{V'} \hookrightarrow \overline{X}$ similarly.
Assume that $X^\infty |_{\overline{V}^N } := \nu_V^\ast X^\infty$ is effective.
Then, the divisor $X^\infty |_{\overline{V'}^N } := \nu_{V'}^\ast X^\infty$ is also effective.
\end{lem}
\begin{proof}
See  \cite[Lemma 2.1]{BS} or \cite[Proposition 2.4]{KP}.
 \end{proof}


Finally, we introduce the tensor product of modulus pairs. 
It will be used in the formulation of the cube invariance.

\begin{dfn}\label{prod}
For modulus pairs $\mathscr{X}=(\overline{X},X^\infty ),\mathscr{Y}=(\overline{Y},Y^\infty )$ over $k$, we define 
\[
\mathscr{X} \otimes \mathscr{Y} := (\overline{X} \times \overline{Y} , X^\infty \times \overline{Y} + \overline{X} \times Y^\infty )
\]
and call it 
the {\it tensor product} of 
$\mathscr{X}$ and $\mathscr{Y}$.
\end{dfn}


\begin{prop}\label{support-prod}
For $\mathscr{X},\mathscr{Y}$ as above, we have
$
(\mathscr{X} \otimes \mathscr{Y})^\circ = \mathscr{X}^\circ \times \mathscr{Y}^\circ .
$
\end{prop}

\begin{proof}
It suffices to prove that \[|X^\infty \times \overline{Y} + \overline{X} \times Y^\infty |=|X^\infty | \times \overline{Y} \cup \overline{X} \times |Y^\infty |.\]
The inclusion $\subset$ is immediate from the definition of support.
To prove the inclusion $\supset$, by symmetry, it is enough to show that  $|X^\infty | \times \overline{Y} \subset |X^\infty \times \overline{Y} + \overline{X} \times Y^\infty |.$
Since both sides are closed subsets of $\overline{X} \times \overline{Y}$, we are reduced to proving that any generic point $\eta \in |X^\infty | \times \overline{Y}$ belongs to $ |X^\infty \times \overline{Y} + \overline{X} \times Y^\infty |$.
Consider the projection $\mathrm{pr}_2 : |X^\infty | \times \overline{Y} \to \overline{Y}$ and set $\overline{\eta} := \mathrm{pr}_2 (\eta ) \in \overline{Y}$.
Since $\mathrm{pr}_2$ is flat, $\overline{\eta}$ is also generic.
By shrinking $\overline{Y}$ around $\overline{\eta}$, we may assume that $Y^\infty = 0$.
By flatness of $\mathrm{pr}_2$,
we obtain $|X^\infty | \times_k \overline{Y} = |X^\infty \times_k \overline{Y}| = |X^\infty \times \overline{Y} + \overline{X} \times Y^\infty |$.
This finishes the proof.
 \end{proof}


\subsection{Definition of the higher Chow group with modulus}

First, we recall the definition of ``cubical" higher Chow group in \cite[\S 4]{B2}.
Let $X$ be an equidimensional scheme of finite type over $k$.
For any integer $r,q \geq 0$, let $\underline{z}^r (X,q)$ denote the free abelian group generated on the following set:
\[
\Set{V | \parbox{95mm}{$V$ is an irreducible closed subset of $X \times \square^q $ of codimension $r$\\ which intersects $X \times F$ properly for any face $F \subset \square^q $}}.
\]
For any $i = 1,\dots q$ and $\epsilon = 0,1$, 
the face map 
\[
d_i^\epsilon : \underline{z}^r (X,q) \to \underline{z}^r (X,q-1)
\]
and the degeneracy map
\[
\pi _i : \underline{z}^r (X,q) \to \underline{z}^r (X,q-1)
\]
are defined as the pullback of algebraic cycles along the inclusion
\begin{align*}
\delta_i^\epsilon &: \square^{q-1} \to \square^q \\
\delta_i^\epsilon (x_1 ,\dots ,x_q ) &= (x_1 ,\dots ,x_{i-1} ,\epsilon ,x_{i+1} ,\dots ,x_q )
\end{align*}
and the projection
\begin{align*}
p_i &: \square^{q} \to \square^{q-1} \\
p_i (x_1 ,\dots ,x_q ) &= (x_1 ,\dots ,\hat{x}_i ,\dots ,x_q )
\end{align*}
respectively.
Then, $\{\underline{z}^r (X,\ast ) ,d^\epsilon_i ,\pi_i \}$ defines a cubical abelian group i.e. a functor 
\[
\mathbf{Cube}^{\mathrm{op}} \to \mathbf{Ab}, \ \underline{q} \mapsto \underline{z}^r(\mathscr{X},q),
\]
where $\mathbf{Ab}$ is the category of abelian groups, and $\mathbf{Cube}$ is the subcategory of $\mathbf{Sets}$ such that the objects are $\underline{q} = \{0,1\}^q , q=0,1,2,\dots $, and the morphisms are generated by the following.
\begin{itemize}
\item[1.] Inclusions: $\delta_i^\epsilon : \underline{q-1} \to \underline{q}, i=1,\dots ,q, \epsilon = 0,1$
\[
\delta_i^\epsilon (\epsilon_1 ,\dots ,\epsilon_q ) = (\epsilon_1 ,\dots ,\epsilon_{i-1} ,\epsilon ,\epsilon_{i+1} ,\dots ,\epsilon_q )
\]
\item[2.] Projections: $p_i : \underline{q} \to \underline{q-1}, i=1,\dots ,n$
\[
p_i (\epsilon_1 ,\dots ,\epsilon_q ) = (\epsilon_1 ,\dots ,\hat{\epsilon }_i ,\dots ,\epsilon_q )
\]
\item[3.]
Permutation of factors: $(\epsilon_1 ,\dots ,\epsilon_q ) \mapsto (\epsilon_{\sigma (1)} ,\dots ,\epsilon_{\sigma (q)} ), \sigma \in S_n$
\item[4.]
Involutions: the maps exchanging $0$ and $1$ at each factor
\end{itemize}
(see  \cite[\S 1.1]{KL}, \cite[\S 1]{Levine} for details).
The following generality on cubical objects is important:

\begin{lem}\label{normalization} (See \cite[Lemma 1.2]{KL})
Let $\mathcal{A}$ be an abelian category and let 
$\underline{A} = \{\underline{A} (\ast ) ,d^\epsilon_i ,\pi_i \}$ be a cubical object in $\mathcal{A}$.
Set 
\begin{align*}
\underline{A} (q)_0
:= \bigcap_{i=1}^q \mathrm{Ker} (d_i^0 ) \subset \underline{A}(q), \ \ 
\underline{A} (q)_{\mathrm{degn}}
:= \sum_{i=1}^q \mathrm{Im}(\pi_i )  \subset \underline{A}(q).
\end{align*}
Define a map $d_q : \underline{A} (q) \to \underline{A} (q)$ as the alternating sum
\[
d_q := \sum_{i=1}^q (-1)^i (d_i^1 - d_i^0 ).
\]
Then, $(\underline{A} (\ast ) ,d_\ast )$ defines a (homological) complex of abelian groups.
Moreover, $d_q$ induces maps on the subgroups: 
\[
\underline{A} (q)_0 \to \underline{A} (q-1)_0 ,\ \ \underline{A} (q)_{\mathrm{degn}} \to \underline{A} (q-1)_{\mathrm{degn}} .
\]
Therefore, we obtain two subcomplexes \[(\underline{A} (\ast )_0 ,d_\ast ) \subset (\underline{A} (\ast ) ,d_\ast ),\ \  (\underline{A} (\ast )_\mathrm{degn} ,d_\ast ) \subset (\underline{A} (\ast ) ,d_\ast )\]
called the {\it reduced subcomplex} and the {\it degenerate subcomplex}, respectively.
We have a canonical splitting of complexes:
\[
\underline{A}(\ast ) = \underline{A}(\ast )_0  \oplus \underline{A} (\ast )_{\mathrm{degn}}.
\]
In particular, the quotient complex 
\[
A (\ast ) := \underline{A}(\ast  ) / \underline{A}(\ast )_{\mathrm{degn}}
\]
is canonically isomorphic to $\underline{A}(\ast )_0 $.
\end{lem}

Since $\underline{z}^r (X,\ast )$ is a cubical abelian group, we obtain subcomplexes 
\[
\underline{z}^r (X,\ast )_{\mathrm{degn}} \text{ and } z^r (X,\ast )_0 .
\]
An element of $\underline{z}^r (X,\ast )_{\mathrm{degn}}$ is called {\it a degenerate cycle}.
The {\it cubical cycle complex} is defined as the quotient complex
\[
z^r (X,\ast ) := \underline{z}^r (X,\ast ) / \underline{z}^r (X,\ast )_{\mathrm{degn}} ,
\]
which is canonically isomorphic to the reduced subcomplex $z^r (X,\ast )_0$.
We define the {\it cubical higher Chow group} of $X$ by
\[
\mathrm{CH}^r (X,q) := \mathrm{H}_q (z^r (X,\ast ) ) \cong \mathrm{H}_q (\underline{z}^r (X,\ast )_0 ).
\]

\begin{rem}
The cubical higher Chow group is canonically isomorphic to the usual (simplicial) higher Chow group (\cite[Theorem 4.3]{B2}).
\end{rem}

We define the higher Chow group {\it with modulus} as the homology group of a subcomplex of the cubical cycle complex.

\begin{dfn}\label{higher-chow-mod}
For a modulus pair $\mathscr{X}=(\overline{X},X^\infty )$ and integers $q,r \geq 0$, we define an abelian subgroup $\underline{z}^r (\mathscr{X}, q) \subset \underline{z}^r (\mathscr{X}^\circ ,q)$ as 
\[
\Set{\sum_i n_i V_i \in \underline{z}^r (\mathscr{X}^\circ ,q) | \parbox{60mm}{each irreducible component $V_i$ satisfies \\ the modulus condition $(\star )$}},
\]
where we say that an irreducible closed subset $V \subset \mathscr{X}^\circ \times \square^q$ satisfies the modulus condition if the following holds:
\begin{itemize}
\item[$(\star )$]
Let $\overline{V}$ be the closure of $V$ in $\overline{X} \times (\mathbb{P}^1 )^q$ and let $\nu_V : \overline{V}^N \to \overline{V} \hookrightarrow \overline{X} \times (\mathbb{P}^1 )^q$ be the composite of the normalization morphism and the natural inclusion.
Then, we have an inequality of Cartier divisors on $\overline{V}^N$:
\[
\nu_V^\ast (X^\infty \times (\mathbb{P}^1 )^q ) \leq  \nu_V^\ast (\overline{X} \times F_q ),
\]
where $F_q$ is an effective Cartier divisor on $(\mathbb{P}^1 )^q$ defined by
\[
F_q = 
\begin{cases} \sum_{i=1}^q (\mathbb{P}^1 )^{i-1} \times \{\infty \} \times (\mathbb{P}^1 )^{q-i} & (q \geq 1), \\
0 & (q=0).
\end{cases}
\]
\end{itemize}
\end{dfn}

\begin{rem}\label{rem1-chow-mod}
For effective modulus pairs, the above definition coincides with the one introduced by Binda-Saito in \cite{BS}.
The case $q \geq 1$ is by definition. We can see the case $q=0$ by noting that for any integral closed subscheme $V \subset \mathscr{X}^\circ $,  $D|_{\overline{V}^N } \leq 0$ iff. $D|_{\overline{V}^N }=0$ iff. $|D| \cap \overline{V} = \emptyset$.
\end{rem}

\begin{prop}\label{chow-cubical}
The structure maps $d_q^\epsilon$ and $p_i$ induces maps 
\[
d_q^\epsilon : \underline{z}^r (\mathscr{X},q) \to \underline{z}^r (\mathscr{X},q-1),  \ \ \pi_i : \underline{z}^r (\mathscr{X},q-1) \to \underline{z}^r (\mathscr{X},q),
\]
and $\{\underline{z}^r (\mathscr{X},\ast ) ,d_i^\epsilon ,\pi_i \}$ defines a cubical subgroup of $\{\underline{z}^r (\mathscr{X}^\circ  ,\ast ) ,d_i^\epsilon ,\pi_i \}$.
\end{prop}

\begin{proof}
Obviously, the permutation of factors and the involutions preserve the face condition and the modulus condition.
So, it suffices to prove that the pullbacks of cycles 
\begin{align*}
d_i^\epsilon = (\delta_i^\epsilon )^\ast : \underline{z}^r(\mathscr{X},q) \to \underline{z}^r(\mathscr{X},q-1), \ \ 
\pi_i = p_i^\ast : \underline{z}^r(\mathscr{X},q-1) \to \underline{z}^r(\mathscr{X},q)
\end{align*}
are well-defined.
The proof for $\pi_i$ is easy, noting that for any integral closed subscheme $V \subset \mathscr{X}^\circ \times \mathbb{A}^{q-1}$, we have
\[
\overline{\pi_i (V)}=\overline{V \times \mathbb{A}^1 } = \overline{V} \times \mathbb{P}^1 \subset \overline{X} \times (\mathbb{P}^1 )^{q-1} \times \mathbb{P}^1 ,
\]
where $\overline{V \times \mathbb{A}^1 } \subset \overline{X} \times (\mathbb{P}^1 )^q$ and $\overline{V} \subset \overline{X} \times (\mathbb{P}^1 )^{q-1}$ are closures. Then, the modulus condition on $V$ immediately implies the modulus condition on $V \times \mathbb{A}^1 $.
The assertion for  $d_i^\epsilon$ is a consequence of Lemma \ref{containment-lemma}.
Indeed, for any irreducible cycle $V \subset \mathscr{X}^\circ \times \mathbb{A}^{q} $, the support $|d_i^\epsilon (V)|$ is contained in $V$, so by Lemma \ref{containment-lemma}, the modulus condition for $V$ implies the modulus condition for $d_i^\epsilon (V)$.
This finishes the proof.
 \end{proof}

\begin{dfn}\label{degenerate-cycle} 
Let $A$ be an abelian group, and set
\[
\underline{z}^r (\mathscr{X},\ast ,A) := \underline{z}^r (\mathscr{X},\ast ) \otimes_\mathbb{Z} A.
\]
Obviously, the cubical structure on $\underline{z}^r (\mathscr{X},\ast)$ induces a cubical structure on $\underline{z}^r (\mathscr{X},\ast ,A)$.
Define the (cubical) {\it cycle complex with modulus with coefficients in $A$} as the quotient complex
\[
z^r(\mathscr{X},\ast ,A) := \underline{z}^r(\mathscr{X},\ast ,A)/\underline{z}^r(\mathscr{X},\ast ,A)_{\mathrm{degn}} .
\]
By Lemma \ref{normalization}, it is canonically isomorphic to the reduced subcomplex \[\underline{z}^r(\mathscr{X},\ast ,A)_0 \subset \underline{z}^r(\mathscr{X},\ast ,A).\]

We define the {\it higher Chow groups with modulus of codimension $r$ of a modulus pair $\mathscr{X}$ with coefficients in $A$} by
\[
\mathrm{CH}^r (\mathscr{X},q,A) := H_q (z^r (\mathscr{X},\ast ,A) ) \cong H_q (\underline{z}^r (\mathscr{X},\ast ,A)_0 ).
\]
When $A=\mathbb{Z}$, we write
\[
\mathrm{CH}^r (\mathscr{X},q):=\mathrm{CH}^r (\mathscr{X},q,\mathbb{Z}).
\]
\end{dfn}


\subsection{A variant}\label{section-variant}
In \S \ref{section-indep}, we also use a variant of the higher Chow group with modulus.

\begin{dfn}\label{variant-higher-chow-mod} 
For a modulus pair $\mathscr{X}=(\overline{X},X^\infty )$ and integers $q,r \geq 0$, we define an abelian subgroup $\underline{z}^r (\mathscr{X}, q)' \subset \underline{z}^r (\mathscr{X}^\circ ,q)$ as 
\[
\Set{\sum_i n_i V_i \in \underline{z}^r (\mathscr{X}^\circ ,q) | \parbox{60mm}{each irreducible component $V_i$ satisfies \\ the na\"{i}ve modulus condition $(\star )'$}},
\]
where we say that an irreducible closed subset $V \subset \mathscr{X}^\circ \times \square^q$ satisfies the na\"{i}ve modulus condition if the following holds:
\begin{itemize}
\item[$(\star )'$] Let $\overline{V}$ be the closure of $V$ in $\overline{X} \times \mathbb{A}^q$ and $\nu_V : \overline{V}^N \to \overline{V} \hookrightarrow \overline{X} \times \mathbb{A}^q$ be the composite of normalization morphism and the natural inclusion.
Then, we have an inequality
\[
\nu_V^\ast (X^\infty \times \mathbb{A}^q ) \leq  0
\]
of Cartier divisors on $\overline{V}^N$.
\end{itemize}
Then, we obtain a cubical abelian subgroup 
\[
\underline{z}^r (\mathscr{X},\ast  )' \subset \underline{z}^r (\mathscr{X}^\circ ,\ast  )
\] by the same argument as in the proof of Proposition \ref{chow-cubical}. 
For any abelian group $A$, we define
\begin{align*}
\underline{z}^r (\mathscr{X},\ast ,A)' &:= \underline{z}^r (\mathscr{X},\ast )' \otimes_\mathbb{Z} A, \\
z^r (\mathscr{X},\ast ,A)' &:= \underline{z}^r (\mathscr{X},\ast ,A)' / \underline{z}^r (\mathscr{X},\ast ,A)_{\mathrm{degn}}', \\
\mathrm{CH}^r (\mathscr{X},q)' &:= H_q (z^r (\mathscr{X},\ast )') \cong H_q (\underline{z}^r (\mathscr{X},\ast )'_0).
\end{align*}
We call $\mathrm{CH}^r (\mathscr{X},q)'$ the {\it na\"{i}ve higher Chow group with modulus}.
\end{dfn}


\begin{rem}\label{effective-case}
If $\mathscr{X}=(\overline{X},X^\infty )$ is effective, then the na\"{i}ve higher Chow group $\mathrm{CH}^r (\mathscr{X},\ast )'$ is independent of the multiplicity of $X^\infty$.
Indeed, for effective $X^\infty$, the na\"{i}ve modulus condition $(\star )'$ is equivalent to the condition that the closure $\overline{V} \subset \overline{X} \times \mathbb{A}^q $ misses the support $|X^\infty \times \mathbb{A}^q |=|X^\infty | \times \mathbb{A}^q $.
\end{rem}

\subsection{Flat pullback and proper pushforward}

\begin{dfn}\label{adm-coadm}
A {\it morphism of modulus pairs} $f : \mathscr{X} \to \mathscr{Y}$ is a morphism of $k$-schemes $\mathscr{X}^\circ \to \mathscr{Y}^\circ$.
We use the symbol $f^\circ$ instead of $f$ if we want to clarify that we are considering just a morphism of $k$-schemes.
The composition of morphisms of modulus pairs are defined by the composition of morphisms of $k$-schemes.
Let $V=\Gamma_{f^\circ } \subset \mathscr{X}^\circ \times \mathscr{Y}^\circ $ denote the graph of $f^\circ $.
Moreover, let $\overline{V} \subset \overline{X} \times \overline{Y}$ be the closure of $V$, and let $\overline{V}^N \to \overline{V}$ be its normalization.
Let $\nu_V : \overline{V}^N \to \overline{V} \hookrightarrow \overline{X} \times \overline{Y}$ be the composite map.
Then, we say 
\begin{itemize}
\item
$f$ is {\it  admissible}  if 
\[
\nu_V^\ast (X^\infty \times \overline{Y} )  \geq \nu_V^\ast (\overline{X} \times Y^\infty ),
\]
\item
$f$ is {\it coadmissible} if
\[
\nu_V^\ast (X^\infty \times \overline{Y} )  \leq \nu_V^\ast (\overline{X} \times Y^\infty ),
\]
\item
$f$ is {\it minimal} if $f$ is admissible and coadmissible,
\item
$f$ is {\it left fine} (resp. {\it right fine}) if $\overline{V}$ is proper over $\overline{X}$ (resp. over $\overline{Y}$),
\item
$f$ is {\it flat} (resp. {\it proper}) if $f^\circ : \mathscr{X}^\circ \to \mathscr{Y}^\circ $ is flat (resp. proper).
\end{itemize}
\end{dfn}

For later use, we introduce the following definition.

\begin{dfn}\label{notation-multiplicity}
For every modulus pair $\mathscr{X}=(\overline{X},X^\infty )$ and an integer $n \in \mathbb{Z} - \{0\}$, define a modulus pair $\mathscr{X}^{(n)}$ by
\[
\mathscr{X}^{(n)} = (\overline{X},n \cdot X^\infty ).
\]
\end{dfn}

\begin{rem}\label{adm-coadm-minus}
(1) For two modulus pairs $\mathscr{X},\mathscr{Y}$, a morphism of modulus pairs $f : \mathscr{X} \to \mathscr{Y}$ is admissible (resp. coadmissible) if and only if $f$ is coadmissble (resp. admissible) considered as a morphism of modulus pairs $\mathscr{X}^{(-1)} \to \mathscr{Y}^{(-1)}$.

(2) A morphism of modulus pairs $f : \mathscr{X} \to \mathscr{Y}$ is left fine (resp. right fine) if there exists a morphism (resp. a proper morphism) of $k$-schemes $\overline{f} : \overline{X} \to \overline{Y}$ extending $f^\circ : \mathscr{X}^\circ \to \mathscr{Y}^\circ $.
Indeed, since $\Gamma_{f^\circ } \subset \Gamma_{\overline{f}}$, we have $\overline{\Gamma }_{f^\circ } \subset \Gamma_{\overline{f}} \cong \overline{X}$.
\end{rem}


We omit the proof of the following proposition because it will not be used in the rest of the paper, and the proof is the same as in \cite[Lemma 1.11]{KSY}.

\begin{prop}\label{composition}
Let $f : \mathscr{X}_1 \to \mathscr{X}_2$ and $g : \mathscr{X}_2 \to \mathscr{X}_3$ be morphisms of modulus pairs and consider the composite $g \circ f : \mathscr{X}_1 \to \mathscr{X}_3$.
Then, the following assertions hold:
\begin{itemize}
\item[(1)]
If $f,g$ are admissible (resp. coadmissible) and left fine, then so is $g \circ f$.
\item[(2)]
If $f,g$ are admissible (resp. coadmissible) and right fine, then so is $g \circ f$.
\end{itemize}
 \end{prop}

\begin{prop}\label{functoriality} Let $A$ be an abelian group.
Let $f : \mathscr{X} \to \mathscr{Y}$ be a morphism of modulus pairs (see Definition \ref{adm-coadm}). 
\begin{itemize}
\item[(1)] Assume that $f$ is coadmissible, left fine and flat as a morphism of modulus pairs. Then, we have a canonical pullback morphisms of cubical groups
\[
f^\ast : \underline{z}^r (\mathscr{Y},\ast ,A) \to \underline{z}^r (\mathscr{X},\ast,A), \ \ f^\ast : \underline{z}^r (\mathscr{Y},\ast ,A)' \to \underline{z}^r (\mathscr{X},\ast ,A)'.
\]
In particular, we obtain a pullback maps for any $q \geq 0$
\[
f^\ast : \mathrm{CH}^r (\mathscr{Y},q,A) \to \mathrm{CH}^r (\mathscr{X},q,A), \ \ f^\ast : \mathrm{CH}^r (\mathscr{Y},q,A)' \to \mathrm{CH}^r (\mathscr{X},q,A)'.
\]

\item[(2)]  Assume that $f$ is admissible, right fine and proper as a morphism of modulus pairs. Then, we have a canonical pushforward morphisms of cubical groups
\[
f_\ast : \underline{z}_r (\mathscr{X},\ast ,A) \to \underline{z}_{r} (\mathscr{Y},\ast ,A), \ \ f_\ast : \underline{z}_r (\mathscr{X},\ast ,A)' \to \underline{z}_{r} (\mathscr{Y},\ast ,A)'.
\]
In particular, we obtain a pushforward maps for any $q \geq 0$
\[
f_\ast : \mathrm{CH}_r (\mathscr{X},q,A) \to \mathrm{CH}_{r} (\mathscr{Y},q,A), \ \ f_\ast : \mathrm{CH}_r (\mathscr{X},q,A)' \to \mathrm{CH}_{r} (\mathscr{Y},q,A)'.
\]
\end{itemize}
\end{prop}

\begin{proof}
For simplicity, we assume that $A=\mathbb{Z}$.
The maps for general $A$ are obtained by taking $\otimes_\mathbb{Z} A$.
Moreover, we only prove that $f^\ast : \underline{z}^r (\mathscr{Y},\ast ) \to \underline{z}^r (\mathscr{X},\ast )$ and $f_\ast : \underline{z}_r (\mathscr{X},\ast ) \to \underline{z}_{r} (\mathscr{Y},\ast )$ are well-defined.
 The proof for the na\"{i}ve variant is obtained by considering $z^r (-,\ast )', \underline{z}^r (-,\ast )'$ instead of $z^r (-,\ast ), \underline{z}^r (-,\ast )$, and replacing $(\mathbb{P}^1 )^q , F_q$ by $\mathbb{A}^q ,0$ respectively

First, we prove (1). Assume that $f$ is coadmissble and flat.
By the definition, the flat morphism $f$ of modulus pairs is just a flat morphism of $k$-schemes $f=f^\circ : \mathscr{X}^\circ \to \mathscr{Y}^\circ$.
This induces the pullback map of cubical abelian groups
$
f^\ast : \underline{z}^r(\mathscr{Y}^\circ ,\ast ) \to \underline{z}^r(\mathscr{X}^\circ ,\ast ).
$
Since we have obvious inclusions $\underline{z}^r(\mathscr{X},\ast ) \subset \underline{z}^r(\mathscr{X}^\circ ,\ast  )$ and $\underline{z}^r(\mathscr{Y},\ast ) \subset \underline{z}^r(\mathscr{Y}^\circ ,\ast  )$, it suffices to prove that the map $f^\ast$ sends $\underline{z}^r(\mathscr{Y},\ast )$ to $\underline{z}^r(\mathscr{X},\ast )$.
In the following we set $X=\mathscr{X}^\circ$ and $Y=\mathscr{Y}^\circ$.

Fix $q\geq 0$ and take any integral algebraic cycle $V \in \underline{z}^r(\mathscr{Y},q)$. 
To show that $W := f^\ast (V)$ belongs to $\underline{z}^r(\mathscr{X},q)$,
it suffices to check the modulus condition.
The support $|f^\ast (V)|=f^{-1}(|V|)$ of $f^\ast (V)$ is the image of 
\[ (\Gamma_f \times \mathbb{A}^q ) \cap (X \times V) \subset X \times Y \times \mathbb{A}^q \]
under the projection
$
X \times Y \times \mathbb{A}^q \to X \times \mathbb{A}^q .
$
Consider the closure 
\[ \overline{(\Gamma_f \times \mathbb{A}^q) \cap (X \times V)}\subset \overline{X} \times \overline{Y} \times (\mathbb{P}^1)^q. \]
Since $f$ is left fine (see Definition \ref{adm-coadm}), $\overline{\Gamma}_f \subset \overline{X} \times \overline{Y}$ is proper over $\overline{X}$, hence the composite map
\[
\overline{\Gamma_f \times \mathbb{A}^q} = \overline{\Gamma_f} \times (\mathbb{P}^1)^q \hookrightarrow \overline{X} \times \overline{Y} \times (\mathbb{P}^1 )^q \to \overline{X} \times (\mathbb{P}^1 )^q
\]
is also proper.
Letting $\overline{W} \subset \overline{X} \times (\mathbb{P}^1 )^q$ denote the closure of $W$, we obtain a proper surjective morphism 
\[
\overline{(\Gamma_f \times \mathbb{A}^q) \cap (X \times V)} \twoheadrightarrow \overline{W}.
\]
By the universality of normalization, we obtain a commutative diagram
\[\xymatrix{
\overline{(\Gamma_f \times \mathbb{A}^q) \cap (X \times V)}^N \ar[r]^(0.72){\exists}_(0.72){\text{surj.}} \ar[d]_{\text{normalization}} & \overline{W}^N \ar[d]^{\text{normalization}} \\
\overline{(\Gamma_f \times \mathbb{A}^q) \cap (X \times V)} \ar[r]_(0.72){\text{surj.}} & \overline{W}.
}\]
Thus, Lemma \ref{effectivity-criterion} implies that the desired inequality of Cartier divisors on $\overline{W}^N$
\[
X^\infty \times (\mathbb{P}^1)^q |_{\overline{W}^N} \leq \overline{X} \times F_q |_{\overline{W}^N}
\]
follows from 
\[
X^\infty \times \overline{Y} \times (\mathbb{P}^1 )^q|_{\overline{(\Gamma_f \times \mathbb{A}^q) \cap (X \times V)}^N} \leq \overline{X} \times \overline{Y} \times  F_q |_{\overline{(\Gamma_f \times \mathbb{A}^q) \cap (X \times V)}^N}.
\]
The last inequality is obtained by the following two inequalities:
\begin{align*}
X^\infty \times \overline{Y} \times (\mathbb{P}^1 )^q|_{\overline{(\Gamma_f \times \mathbb{A}^q) \cap (X \times V)}^N} &\stackrel{(a)}{\leq} \overline{X} \times Y^\infty \times  (\mathbb{P}^1 )^q |_{\overline{(\Gamma_f \times \mathbb{A}^q) \cap (X \times V)}^N} , \\
\overline{X} \times Y^\infty \times (\mathbb{P}^1 )^q|_{\overline{(\Gamma_f \times \mathbb{A}^q) \cap (X \times V)}^N} &\stackrel{(b)}{\leq} \overline{X} \times \overline{Y} \times  F_q |_{\overline{(\Gamma_f \times \mathbb{A}^q) \cap (X \times V)}^N}.
\end{align*}
By Lemma \ref{containment-lemma}, (a) and (b) follow from
\begin{align*}
X^\infty \times \overline{Y} \times (\mathbb{P}^1 )^q|_{\overline{\Gamma_f \times \mathbb{A}^q}^N} &\stackrel{(a)'}{\leq} \overline{X} \times Y^\infty \times  (\mathbb{P}^1 )^q |_{\overline{\Gamma_f \times \mathbb{A}^q}^N} , \\
\overline{X} \times Y^\infty \times (\mathbb{P}^1 )^q|_{\overline{X \times V}^N} &\stackrel{(b)'}{\leq} \overline{X} \times \overline{Y} \times  F_q |_{\overline{X \times V}^N} .
\end{align*}
respectively. 
Since $\overline{\Gamma_f \times \mathbb{A}^q }=\overline{\Gamma}_f \times (\mathbb{P}^1 )^q$,
the coadmissibility of $f$ implies $(a)'$.
On the other hand, since $\overline{X \times V}=\overline{X} \times \overline{V}$, the assumption on $V$ implies $(b)'$.
This finishes the proof of (1).


Next, we prove (2). 
Set $X=\mathscr{X}^\circ$ and $Y=\mathscr{Y}^\circ$.
Since $f$ is proper we have the pushforward map of cubical abelian groups $f_\ast : \underline{z}_r (X,\ast )  \to \underline{z}_r (Y,\ast )$.
Noting that 
$\underline{z}_r(\mathscr{X},\ast ) \subset \underline{z}_r(X ,\ast  )$ and $\underline{z}_r(\mathscr{Y},\ast ) \subset \underline{z}_r(Y ,\ast  )$, it suffices to prove that the map $f_\ast$ sends $\underline{z}_r(\mathscr{X},\ast )$ to $\underline{z}_r(\mathscr{Y},\ast )$.

Fix $q \geq 0$ and take any prime cycle $V \in \underline{z}_r (\mathscr{X},q)$.
We check that 
\[ f_\ast (V) \in \underline{z}_r (\mathscr{Y},q). \]
Set $W:=|f_\ast (V)| = f(V) \subset Y \times \mathbb{A}^q$. Since $V$ is prime, $W$ is irreducible. Thus, it suffices to prove that
\[Y^\infty |_{\overline{W}^N} \leq F_q |_{\overline{W}^N},\] where $\overline{W} \subset \overline{Y} \times (\mathbb{P}^1 )^q$ is the closure of $W$. 
Consider the following commutative diagram
\[\xymatrix{
X \times  Y \times \mathbb{A}^q  \ar[r] & Y  \times \mathbb{A}^q  \\
\Gamma_f  \times \mathbb{A}^q  \cap (V \times Y) \ar[r]^(0.55){\mathrm{surj.}} \ar[u]^{\mathrm{inclusion}} & f(V)=W \ar[u]_{\mathrm{inclusion}}
}\]
where $\Gamma_f$ is the graph of $f$ and the bottom line is surjective by a trivial reason. 
Then, since $\overline{\Gamma_f \times \mathbb{A}^q } \subset \overline{X} \times \overline{Y} \times (\mathbb{P}^1 )^q $ is proper over $\overline{Y} \times (\mathbb{P}^1 )^q $ by the assumption that $f$ is right fine, 
the induced map \[\overline{X} \times \overline{Y} \times (\mathbb{P}^1 )^q  \supset \overline{\Gamma_f \times \mathbb{A}^q } \cap (V \times Y) \to \overline{f(V)}=\overline{W} \subset \overline{Y}\times (\mathbb{P}^1 )^q \] is also proper, hence surjective. Thus, we can obtain the desired inequaly of Cartier divisors on $\overline{W}^N$ by following the argument of proof of (1).
 \end{proof}

\begin{cor}\label{norm-map}
Let $k'/k$ be a finite field extension and let $\pi : \mathscr{X}_{k'} \to \mathscr{X}$ be the induced morphism, where $\mathscr{X}_{k'} := (\overline{X}_{k'} , X^\infty_{k'} )$.
Then, the composite maps 
\begin{align*}
\pi_\ast \circ \pi^\ast &: \underline{z}^r (\mathscr{X},\ast ,A) \to \underline{z}^r (\mathscr{X}_{k'} ,\ast  ,A) \to \underline{z}^r (\mathscr{X},\ast  ,A), \\
\pi_\ast \circ \pi^\ast &: \underline{z}^r (\mathscr{X},\ast  ,A)' \to \underline{z}^r (\mathscr{X}_{k'} ,\ast  ,A)' \to \underline{z}^r (\mathscr{X},\ast  ,A)'
\end{align*}
equal the multiplication by $[k':k]$.
\end{cor}

\section{Cube invariance}\label{section-cube-invariance}

The aim of this section is to prove the {\it cube invariance} (Thoeorem \ref{main-1}).
Recall that $\overline{\square}^{(-1)} = (\mathbb{P}^1 ,-\{\infty \})$.
Note that for any modulus pair $\mathscr{X}$, the first projection $\mathrm{pr}_1 : \mathscr{X}^\circ  \times \mathbb{A}^1 \to \mathscr{X}^\circ$ defines a morphism of modulus pairs 
\[
\mathrm{pr}_1 : \mathscr{X}  \otimes \overline{\square}^{(-1)} \to \mathscr{X}
\]
since $(\overline{\square}^{(-1)} )^\circ = \mathbb{A}^1$.
Moreover, $\mathrm{pr}_1$ is left fine, flat, and coadmissible because it extends to the map $\mathrm{pr}_1 : \overline{X} \times \mathbb{P}^1 \to \overline{X}$ and we have $\mathrm{pr}_1^\ast X^\infty = X^\infty \times \mathbb{P}^1$.

\ 

Thoeorem \ref{main-1} is a consequence the following theorem.
We prove it in \S \ref{pf-of-inv}.

\begin{thm}\label{cube-invariance}
Let $\mathscr{X}=(\overline{X},X^\infty )$ be a modulus pair, and let $A$ be an abelian group.
Then, the pullback maps
\begin{align*}
\mathrm{pr}_1^\ast &: z^r (\mathscr{X}, \ast ,A) \to z^r (\mathscr{X} \otimes \overline{\square}^{(-1)}, \ast ,A), \\
\mathrm{pr}_1^\ast &: z^r (\mathscr{X}, \ast ,A)' \to z^r (\mathscr{X} \otimes \overline{\square}^{(-1)}, \ast ,A)'
\end{align*}
are quasi-isomorphisms.
\end{thm}

\begin{rem}\label{reason-generalization}
We can prove that for any modulus pairs $\mathscr{X}$ and $\mathscr{Y}$ such that $Y^\infty $ is effective (and non-trivial), the pullback map
\[
\underline{z}^r (\mathscr{X},\ast ) \to \underline{z}^r (\mathscr{X} \otimes \mathscr{Y},\ast )
\]
is {\it not} well-defined.
This fact shows that we need non-effective modulus pairs in order to formulate a generalization of $\mathbb{A}^1$-homotopy invariance.

\end{rem}

\subsection{Rigidity lemma}

Let $X$ be an equidimensional scheme of finite type over $k$, and let $w$ be a finite set of locally closed subsets of $X$.
For any integers $r,q \geq 0$, we define a subgroup
\[
\underline{z}^r_w (X,q) \subset \underline{z}^r (X,q)
\]
as the free abelian group generated on the set 
\[
\Set{V | \parbox{100mm}{$V$ is an irreducible closed subset of $X \times \square^q $ of codimension $r$\\ which intersects $W \times F$ properly for any face $F \subset \square^q $ and $W \in w$}}.
\]
Since the structure maps $d_i^\epsilon : \underline{z}^r (X,q) \to \underline{z}^r (X,q-1)$ and $\pi_i : \underline{z}^r (X,q-1) \to \underline{z}^r (X,q)$ restrict to the subgroups, we obtain a cubical subgroup
\[
\underline{z}^r_w (X,\ast ) \subset \underline{z}^r (X,\ast ).
\]
In particular, we have a subcomplex 
\[
z^r_w (X,\ast ) = \underline{z}^r_w (X,\ast ) / \underline{z}^r_w (X,\ast )_\mathrm{degn} \subset z^r (X,\ast ).
\]
This subcomplex plays an important role in Bloch's proof of $\mathbb{A}^1$-homotopy invariance of the higher Chow group \cite[Theorem 2.1]{B}, which leads us to the following definition.

\begin{dfn}\label{nice-subgroup}
Let $\mathscr{X}=(\overline{X},X^\infty )$ be a modulus pair and let $w$ be a finite set of locally closed subsets of $\mathscr{X}^\circ$.
Define cubical subgroups $\underline{z}^r_w (\mathscr{X},\ast ) \subset \underline{z}^r(\mathscr{X} ,\ast )$ and $\underline{z}^r_w (\mathscr{X},\ast )' \subset \underline{z}^r(\mathscr{X} ,\ast )'$ by 
\begin{align*}
\underline{z}^r_w (\mathscr{X},\ast ) &:= \underline{z}^r(\mathscr{X},\ast ) \cap \underline{z}^r_w (\mathscr{X}^\circ ,\ast ), \\
\underline{z}^r_w (\mathscr{X},\ast )' &:= \underline{z}^r(\mathscr{X},\ast )' \cap \underline{z}^r_w (\mathscr{X}^\circ ,\ast ),
\end{align*}
where the intersections are taken in $\underline{z}^r (\mathscr{X}^\circ ,\ast )$.

For any abelian group $A$, we set
\begin{align*}
\underline{z}^r_w (\mathscr{X},\ast ,A) &:= \underline{z}^r_w (\mathscr{X},\ast ) \otimes_\mathbb{Z} A, \ \ 
\underline{z}^r_w (\mathscr{X},\ast ,A)' := \underline{z}^r_w (\mathscr{X},\ast )' \otimes_\mathbb{Z} A, \\
z^r_w (\mathscr{X},\ast ,A) &:= \underline{z}^r_w (\mathscr{X},\ast ,A) / \underline{z}^r_w (\mathscr{X},\ast ,A)_\mathrm{degn} \cong \underline{z}^r_w (\mathscr{X},\ast ,A)_0 , \\
z^r_w (\mathscr{X},\ast ,A)' &:= \underline{z}^r_w (\mathscr{X},\ast ,A)' / \underline{z}^r_w (\mathscr{X},\ast ,A)'_\mathrm{degn} \cong \underline{z}^r_w (\mathscr{X},\ast ,A)'_0 .
\end{align*}
Note that we have natural inclusions
\begin{align*}
z^r_w (\mathscr{X},\ast ,A) &\subset z^r (\mathscr{X},\ast ,A),\\
z^r_w (\mathscr{X},\ast ,A)' &\subset z^r (\mathscr{X},\ast ,A)'.
\end{align*}
\end{dfn}

Now, consider the tensor product 
\[\mathscr{X} \otimes \overline{\square}^{(-1)} = (\overline{X} \times \mathbb{P}^1 , X^\infty \times \mathbb{P}^1 - \overline{X} \times \{\infty \})\]
of modulus pairs $\mathscr{X}=(\overline{X},X^\infty )$ and $\overline{\square}^{(-1)}=(\mathbb{P}^1 ,-\{\infty \})$. Set
\[
w =\{ \mathscr{X}^\circ \times \{0,1\} \}, \ \ \mathscr{X}^\circ \times \{0,1\} \subset \mathscr{X}^\circ \times \mathbb{A}^1 .
\]
Let $\epsilon \in \{0,1\}$.
By Lemma \ref{containment-lemma},
noting that the inclusions
\begin{align*}
i_\epsilon &: \mathscr{X} \cong \mathscr{X} \otimes \{\epsilon \} \hookrightarrow \mathscr{X} \otimes \overline{\square}^{(-1)}, \\
\overline{i}_\epsilon &: \mathscr{X} \cong \mathscr{X} \otimes \{\epsilon \} \hookrightarrow \mathscr{X} \otimes \mathbb{A}^1, 
\end{align*}
are minimal (see Definition \ref{adm-coadm}), we obtain well-defined pullback maps of complexes 
\begin{align*}
i_\epsilon^\ast &: z^r_w (\mathscr{X} \otimes \overline{\square}^{(-1)} ,\ast ,A) \to z^r(\mathscr{X},\ast ,A), \\
\overline{i}_\epsilon^\ast &: z^r_w (\mathscr{X} \otimes \mathbb{A}^1 ,\ast ,A)' \to z^r(\mathscr{X},\ast ,A)' .
\end{align*}
We have the following lemma.
\begin{lem}\label{rigidity} {\bf (rigidity Lemma)}
The maps $i_0^\ast$ and $i_1^\ast$ are homotopic.
Moreover, the maps $\overline{i}_0^\ast$ and $\overline{i}_1^\ast$ are homotopic.
\end{lem}

In particular, the pullback maps induced by $i_\epsilon$ on na\"{i}ve version
\[
i_\epsilon^\ast : z^r_w (\mathscr{X} \otimes \overline{\square}^{(-1)} ,\ast ,A)' \to z^r(\mathscr{X},\ast ,A)' , \ \ (\epsilon \in \{0,1\})
\]
are also homotopic, since it factors as \[z^r_w (\mathscr{X} \otimes \overline{\square}^{(-1)} ,\ast ,A)' \hookrightarrow z^r_w (\mathscr{X} \otimes \mathbb{A}^1 ,\ast  , A)' \xrightarrow{\overline{i}_\epsilon^\ast } z^r(\mathscr{X},\ast ,A)',\] 
where the first map is the inclusion.

\begin{proof}
We may assume that $A=\mathbb{Z}$.
Note that the identity map
\[
(\mathscr{X}^\circ \times \mathbb{A}^1 ) \times \mathbb{A}^q \stackrel{=}{\to} \mathscr{X}^\circ \times (\mathbb{A}^1 \times \mathbb{A}^q )
\]
induces the canonical identifications
\begin{align*}
\tilde{\Phi}_q &: \underline{z}^r_w (\mathscr{X} \otimes \overline{\square}^{(-1)} ,q ) \xrightarrow{=} \underline{z}^r(\mathscr{X},q+1), \\
\tilde{\overline{\Phi}}_q &: \underline{z}^r_w (\mathscr{X} \otimes \mathbb{A}^1 ,q )' \xrightarrow{=} \underline{z}^r(\mathscr{X},q+1)' ,
\end{align*}
for every $q$.
Since these maps send a degenerate cycle to a degenerate cycle, they induce
\begin{align*}
\Phi_q &: z^r_w (\mathscr{X} \otimes \overline{\square}^{(-1)} ,q ) \to z^r(\mathscr{X},q+1), \\
\overline{\Phi}_q &: z^r_w (\mathscr{X} \otimes \mathbb{A}^1 ,q )' \to z^r(\mathscr{X},q+1)'.
\end{align*}
Then, these maps satisfy the homotopy relations
\begin{align*}
d_{q+1} \circ \Phi_q -\Phi_{q-1} \circ d_q &= i_1^\ast - i_0^\ast  ,\\
d_{q+1} \circ \overline{\Phi}_q -\overline{\Phi}_{q-1} \circ d_q &= \overline{i}_1^\ast - \overline{i}_0^\ast .
\end{align*}
This finishes the proof.
 \end{proof}

\subsection{A moving lemma: powered generic translation}
We prove the following key result, which is used frequently in the rest of the paper.
Fix an algebraic closure $\overline{k}$ of the base field $k$.

\begin{thm}\label{moving-lemma}{\bf (moving lemma)}
Let $\mathscr{Y}=(\overline{Y},Y^\infty )$ be a modulus pair.
Suppose that we are given the following data:
\begin{itemize}
\item[(1)]a connected algebraic $k$-group $G$ acting on $\overline{Y}$ such that the map
\[
\gamma : G \times \overline{Y} \to \overline{Y}, \ \ (y,g) \mapsto g \cdot y
\]
satisfies that $\gamma^\ast (Y^\infty )=G \times Y^\infty$. Note that this implies automatically that $\gamma$ restricts to 
\[
\gamma^\circ : G \times \mathscr{Y}^\circ \to \mathscr{Y}^\circ.
\]

\item[(2)]
A finitely generated purely transcendental field extension $K \supset k$ and 
a $K$-morphism $\psi : \mathbb{A}^1_K \to G_K$ such that $\psi (0) = 1_{G_K }$ and 
for any $\tau \in \mathbb{A}^1 (\overline{k})$ with $\tau \neq 0$, 
the image of the composite
\[
\mathrm{pr} \circ  \psi \circ  (K \otimes_k \tau ) : 
\mathrm{Spec}(K \otimes_k \overline{k})
\xrightarrow{K \otimes_k \tau}
\mathbb{A}^1_K
\xrightarrow{\psi }
G_K
\xrightarrow{\mathrm{pr}} G
\]
is the generic point of $G$, where $\mathrm{pr}$ is the natural projection.

\item[(3)] A finite collection $w$  of irreducible locally closed subsets of $\mathscr{Y}^\circ = \overline{Y}-|Y^\infty |$ such that 
for any $W \in w$,
the map $G \times W \to \mathscr{Y}^\circ $ is surjective and all fibers have the same dimension.
\end{itemize}
Then, the following inclusions are quasi-isomorphisms:
\[
z^r_w (\mathscr{Y} ,\ast ,A) \hookrightarrow z^r(\mathscr{Y} ,\ast ,A),
\ \ z^r_w (\mathscr{Y} ,\ast ,A)' \hookrightarrow z^r(\mathscr{Y} ,\ast ,A)'.
\]
\end{thm}

\begin{rem}\label{remark-normalized}
Note that the theorem shows that 
\[
\underline{z}^r_w (\mathscr{Y} ,\ast ,A)_0 \hookrightarrow \underline{z}^r(\mathscr{Y} ,\ast ,A)_0 ,
\ \ \underline{z}^r_w (\mathscr{Y} ,\ast ,A)'_0 \hookrightarrow \underline{z}^r(\mathscr{Y} ,\ast ,A)'_0 
\]
are also quasi-isomorphisms (see Lemma \ref{normalization}).
\end{rem}

\begin{cor}\label{special-moving-lemma} 
Let $\overline{X}=(\overline{X},X^\infty )$ be a modulus pair over $k$.
For integers $n \geq 0$ and $m_1 ,\dots ,m_n \in \mathbb{Z}$, we define a modulus pair $\mathscr{Y}$ by 
\[
\mathscr{Y} := \mathscr{X} \otimes \bigotimes_{i=1}^{n} \overline{\square }^{(m_i )} ,
\]
where $\overline{\square }^{(m)} = (\mathbb{P}^1 ,m \cdot \{\infty \})$ if $m \in \mathbb{Z}-\{0\}$ and $\overline{\square }^{(0)} := (\mathbb{A}^1 ,\emptyset ) = \mathbb{A}^1$.
Consider the product of algebraic groups $G := (\mathbb{G}_a )^{n}$ and the action $G \curvearrowright \mathscr{Y}^\circ = \mathscr{X}^\circ \times \mathbb{A}^n $ defined for $\bvec{t}=(t_1 ,\dots ,t_n ) \in G$ and for $y=(x,s_1 ,\dots ,s_n ) \in \mathscr{X}^\circ \times \mathbb{A}^n $ by
\[
\bvec{t} \cdot y := (x,s_1 + t_1 ,\dots ,s_n + t_n ).
\]
Let $w$ be a finite collection
of faces of $\mathscr{Y}^\circ = \mathscr{X}^\circ \times \mathbb{A}^n$.
Then, the following inclusions are quasi-isomorphisms:
\[
z^r_w (\mathscr{Y} ,\ast ,A) \hookrightarrow z^r(\mathscr{Y} ,\ast ,A), \ \ 
z^r_w (\mathscr{Y} ,\ast ,A)' \hookrightarrow z^r(\mathscr{Y} ,\ast ,A)'.
\]
\end{cor}

\begin{proof}
First, note that the action 
$G \curvearrowright \mathscr{Y}^\circ = \mathscr{X}^\circ \times \mathbb{\mathbb{A} }^n $
naturally extends to an action
$
G  \curvearrowright \overline{Y},
$
since the additive group $\mathbb{G}_a$ acts on $\mathbb{P}^1$ by 
$t \cdot ([S_0 : S_1]) = [S_0 : S_1 + t \cdot S_0 ]$ (This action is due to Krishna-Park).
Consider the map $\gamma : G \times \overline{Y} \to \overline{Y}, (\bvec{t}, y) \mapsto \bvec{t} \cdot y$.
Then, $\gamma^\ast (Y^\infty ) = G \times Y^\infty$ trivially holds.
Let $K=k(t_1 ,\dots ,t_n )$ be the function field of $G=(\mathbb{G}_a )^n$ and consider a $K$-morphism
\[
\mathbb{A}^1_K \to G_K = (\mathbb{G}_a )^n_K , \ t \mapsto t \cdot \bvec{v},
\]
where $\bvec{v}=(v_1 ,\dots ,v_n ) \in G$ is the generic point of $G$ and $t \cdot \bvec{v}=(tv_1 ,\dots ,tv_n )$.
Then, we can see that for any closed point $\tau \in \mathbb{A}^1 - \{0\}$, the point $\tau \cdot \bvec{v}$ lies over a generic point of $G$.
Indeed, the $K(\tau )( := K \otimes_k k(\tau))$-valued point $\tau \cdot \bvec{v} \to G_{K(\tau )}$ corresponds to the surjective $k$-homomorphism $K(\tau )[t_1 ,\cdots ,t_n ] \to K(\tau )$ which sends each $t_i$ to $\tau \cdot t_i \in K(\tau )$.
We want to prove that the image of the composite $\tau \cdot \bvec{v} \to G_{K(\tau )} \to G$ is generic.
This is equivalent to that the kernel of the composite $c : k[t_1 ,\dots ,t_n ] \to K(\tau )[t_1 ,\cdots ,t_n ] \to K(\tau )$ is zero.
Since $\tau $ is algebraic over $k$, the images $c(t_1 ) = \tau \cdot t_1 ,\dots ,c(t_n )=\tau \cdot t_n \in K(\tau )$ have no algebraic relation over $k$, which shows the desired statement.
Thus, we can apply Theorem \ref{moving-lemma}.
This finishes the proof.
 \end{proof}

\

In the following, we prove Theorem \ref{moving-lemma}.
Since the following argument is independent of $A$, we assume that $A=\mathbb{Z}$ without loss of generality, in order to simplify the notation.
Moreover, using Corollary \ref{norm-map}, we may assume that the base field is {\it infinite} by considering a tower of field extensions of degree prime to a given prime $p$.

Noting that $\underline{z}^r (\mathscr{Y},\ast ) \cong z^r (\mathscr{Y},\ast ) \oplus \underline{z}^r (\mathscr{Y},\ast )_{\mathrm{degn}}$ and $\underline{z}^r_w (\mathscr{Y} ,\ast ) \cong z^r_w (\mathscr{Y},\ast ) \oplus \underline{z}^r_w (\mathscr{Y},\ast )_{\mathrm{degn}}$ (and the similar formulae for the na\"{i}ve variant), it suffices to prove that the inclusions
\[
\underline{z}^r_w (\mathscr{Y},\ast ) \hookrightarrow \underline{z}^r (\mathscr{Y},\ast ),
\ \ \underline{z}^r_w (\mathscr{Y},\ast )' \hookrightarrow \underline{z}^r (\mathscr{Y},\ast )'
\]
are quasi-isomorphisms of complexes.
This is equivalent to show that the quotient complex $\underline{z}^r (\mathscr{Y} ,\ast  )^{(\prime )} / \underline{z}^r_w (\mathscr{Y} ,\ast  )^{(\prime )}$ is acyclic.
Consider the projection $\pi : \mathscr{Y}_K \to \mathscr{Y}$, where ${-}_K = - \otimes_k K$ denotes base extension, where
\[\mathscr{Y}_K = (\overline{Y}_K ,Y^\infty_K ) = (\overline{Y} \otimes_k K ,Y^\infty  \otimes_k K).\] Then, $\pi$ induces a canonical pullback maps \[\pi^\ast : \underline{z}^r (\mathscr{Y},\ast )^{(\prime )} \to \underline{z}^r (\mathscr{Y}_K ,\ast )^{(\prime )} .\]
To prove that $\underline{z}^r (\mathscr{Y} ,\ast  )^{(\prime )} / \underline{z}^r_w (\mathscr{Y} ,\ast  )^{(\prime )}$ is acyclic, it suffices to prove the following assertions:

\begin{prop}\label{key-prop-1-2}
(1)
The induced map
\[
\overline{\pi}^\ast : \underline{z}^r (\mathscr{Y} ,\ast  )^{(\prime )} / \underline{z}^r_w (\mathscr{Y} ,\ast  )^{(\prime )} \to \underline{z}^r (\mathscr{Y}_K ,\ast  )^{(\prime )} / \underline{z}^r_w (\mathscr{Y}_K ,\ast  )^{(\prime )}
\]
on quotient complexes is injective on homology groups.

(2)
The map $\overline{\pi}^\ast$ is zero on homology groups.
\end{prop}

To prove the assertion (1), we need the following lemma.

\begin{lem}\label{colimit}
Let $X$ be a scheme of finite type over a field $k$ and $D$ a Cartier divisor on $X$.
Let $w$ be a finite set of closed subsets of $X \setminus |D|$.
Suppose given a filtered projective system $(U_i )_{i \in I}$ of open subsets of $X$ such that the transition morphisms $\iota_i^j : U_j \hookrightarrow U_i$ are affine for all $i<j$.
Set $\widehat{U} := \cap_{i \in I} U_i$.
Moreover, let $\widehat{D}$ (resp. $D_i$) be the pullback of $D$ onto $\widehat{U}$ (resp. $U_i$).
Then, for any $r,q\geq 0$, we have canonical isomorphisms
\[
\mathrm{colim}_{i \in I} \underline{z}^r_w (U_i |D_i ,q) \xrightarrow{\cong }  \underline{z}^r_w (\widehat{U} |\widehat{D} ,q),\ \ \mathrm{colim}_{i \in I} \underline{z}^r_w (U_i |D_i ,q)' \xrightarrow{\cong }  \underline{z}^r_w (\widehat{U} |\widehat{D} ,q)',
\]
where the subscript $w$, on the left (resp. right) hand sides of the isomorphisms, denotes the set $\{W \cap U_i |W \in w \}$ (resp. $\{W \cap \widehat{U}| W \in w \}$).

\end{lem}

\begin{proof}
We give a proof of the first isomorphism.
The proof for the na\"{i}ve version is formally obtained by replacing every $z^r (-,\ast ) , z^r_w (-,\ast ), (\mathbb{P}^1 )^q, F_q$ by $z^r (-,\ast )' , z^r_w (-,\ast )' , \mathbb{A}^q ,0$ respectively.

First, we prove injectivity.
Set $U_i^\circ := U_i \setminus |D_i |$ for each $i$, and set $\widehat{U}^\circ := \widehat{U} \setminus \widehat{D}$.
Then, we have $\widehat{U}^\circ = \cap_i U_i^\circ$.
It is satandard that $\mathrm{colim}_{i} \underline{z}^r_w (U_i^\circ ,q) \xrightarrow{\cong }  \underline{z}^r_w (\widehat{U}^\circ ,q)$.
Since $\underline{z}^r (U_i |D_i ,q) \subset \underline{z}^r (U_i^\circ ,q)$ and $\underline{z}^r (\widehat{U} |\widehat{D} ,q) \subset \underline{z}^r (\widehat{U}^\circ ,q)$ by definiton, the injectivity follows.
Next, we prove surjectivity.
Take any integral cycle $Z \in \underline{z}^r (\widehat{U}|\widehat{D},q) \subset \underline{z}^r (\widehat{U}^\circ ,q)$.
By the isomorphism $\mathrm{colim}_{i} \underline{z}^r (U_i^\circ ,q) \xrightarrow{\cong }  \underline{z}^r (\widehat{U}^\circ ,q)$, there exist an element $i_0 \in I$ and a cycle $Z_i \in \underline{z}^r (U_{i_0}^\circ ,q)$ such that the restriction of $Z_{i_0}$ coincides with $Z$.
By discarding components of $Z_{i_0}$ if necessary, we may assume that $Z_{i_0}$ is also integral.
Then, $Z_{i_0}$ equals the closure of $Z$ in the scheme $U_{i_0}^\circ \times \mathbb{A}^q$.
Consider the closed subset \[A =\overline{Z}_{i_0}^N \setminus \Set{x \in \overline{Z}_{i_0}^N | D \times (\mathbb{P}^1 )^q |_{\overline{Z}_{i_0}^N} \leq  X \times F_q |_{\overline{Z}_{i_0}^N} \text{ around } x}\] of $\overline{Z}_{i_0}^N$.
It suffices to prove that $A_i := A_{i_0} \times_{U_0 } U_i = \emptyset$ for some $i > i_0$.
Let $\overline{Z}$ be the closure of $Z$ in $\widehat{U} \times (\mathbb{P}^1 )^q$, and let $\overline{Z}^N \to \overline{Z}$ be the normalization.
Then, $\overline{Z}^N = \overline{Z}_{i_0 }^N \times_{U_{i_0}} \widehat{U}$.
By the assumption that $Z \in \underline{z}^r (\widehat{U}|\widehat{D},q)$, we have $\overline{Z}^N \supset A_{i_0 } \times_{U_{i_0}} \widehat{U} = \emptyset$.
Thus, $\overline{Z}^N$ does not contain any generic point of $A_{i_0}$, which implies that there exists $i>i_0$ such that $A_i = \emptyset$.
This finishes the proof.
 \end{proof}

We prove the assertion (1) by using the standard ``specialization argument" as in the proof of \cite[Lemma 2.3]{B}.
Since $k \subset K$ is a finitely generated purely transcendental extension, we may assume that $\mathrm{tr.deg}_k (K)=1$ by induction on the transcendental degree.
Then, $K=k(t)$ for some transcendental element $t \in K$.
By Lemma \ref{colimit}, we have $\underline{z}^r (\mathscr{Y}_K ,q)^{(\prime )} = \mathrm{colim}_{U} \underline{z}^r (\mathscr{Y} \times U,q)^{(\prime )}$ and $\underline{z}^r_w (\mathscr{Y}_K ,q)^{(\prime )} = \mathrm{colim}_{U} \underline{z}^r_w (\mathscr{Y} \times U,q)^{(\prime )}$ for any $q \geq 0$, where $U$ runs over non-empty open subsets of $\mathbb{A}^1 = \mathrm{Spec}(k[t])$.
Fix $q \geq 0$, take any element $[Z] \in \mathrm{Ker} (\overline{\pi }) \subset \underline{z}^r (\mathscr{Y} ,q )^{(\prime )} / \underline{z}^r_w (\mathscr{Y} ,q  )^{(\prime )}$, and let $Z \in \underline{z}^r (\mathscr{Y} ,q )^{(\prime )}$ be a lift of $[Z]$.
Then, \[\pi^\ast (Z) \in  \underline{z}^r_w (\mathscr{Y}_K ,q )^{(\prime )} + \mathrm{Im}[d_{q+1} : \underline{z}^r (\mathscr{Y}_K ,q+1 )^{(\prime )} \to  \underline{z}^r (\mathscr{Y}_K ,q )^{(\prime )} ].\]
We can find a non-empty open subset $U \subset \mathbb{A}^1$ such that 
$
Z \times U = Z_1 + Z_2,
$
where $Z_1 \in \underline{z}^r_w (\mathscr{Y} \times U ,q )^{(\prime )}$ and $Z_2 \in \mathrm{Im}[d_{q+1} : \underline{z}^r (\mathscr{Y} \times U ,q+1 )^{(\prime )} \to  \underline{z}^r (\mathscr{Y} \times U ,q )^{(\prime )} ]$.
By shrinking $U$ if necessary, we may assume that $Z_1$ and $Z_2$ are equidimensional over $U$ (by the semi-continuity of the dimension function).
Since the field $k$ is assumed to be infinite, there exists a $k$-rational point $x : \mathrm{Spec}(k) \to U$ of $U$.
By specializing the above equation at $x$ (which is possible by equidimensionality), we obtain $Z = Z_{1,x} + Z_{2,x}$ with $Z_{1,x} \in \underline{z}^r_w (\mathscr{Y} ,q )^{(\prime )}$ and $Z_{2,x} \in \mathrm{Im}[d_{q+1} : \underline{z}^r (\mathscr{Y} ,q+1 )^{(\prime )} \to  \underline{z}^r (\mathscr{Y} ,q )^{(\prime )} ]$.
This means that $[Z]=0$.

\ 

This finishes the proof of the assertion (1) of Proposition \ref{key-prop-1-2}.

\subsection{Construction of homotopy}

To prove the assertion (2) of Proposition \ref{key-prop-1-2}, it suffices to construct a weak homotopy between $\overline{\pi}^\ast$ and the zero map.
Consider a ``generic translation" $K$-isomorphism 
\[
\overline{\varphi} : \overline{Y}_K \times \mathbb{A}^1 \to \overline{Y}_K \times \mathbb{A}^1, \ \ (y, t) \mapsto (\gamma (\psi (t),y),t) = (\psi (t) \cdot y ,t),
\]
which satisfies $\overline{\varphi }^\ast (Y^\infty_K \times \mathbb{A}^1 ) = Y^\infty_K \times \mathbb{A}^1$.
The inverse is given by 
\[(y,t) \mapsto (\gamma (\psi (t)^{-1} ,y) ,t)=(\psi (t)^{-1} \cdot y,t).\]
Note that $\overline{\varphi }$ restricts to a $K$-isomorphism
\[
\varphi :  \mathscr{Y}^\circ_K \times \mathbb{A}^1 \to \mathscr{Y}^\circ_K \times \mathbb{A}^1, \ \ (y,t) \mapsto (\gamma^\circ (\psi (t),y) ,t)=(\psi (t)\cdot y,t).
\]
Fix $q \geq 0$, and consider the composite
\begin{align*}
\Psi : (\mathscr{Y}^\circ_K \times \square^q ) \times \mathbb{A}^1 &\cong (\mathscr{Y}^\circ_K \times \mathbb{A}^1 ) \times \square^q \\ &\xrightarrow{\varphi \times \mathrm{id}_{\square^q }} (\mathscr{Y}^\circ_K \times \mathbb{A}^1 ) \times \square^q \cong (\mathscr{Y}^\circ_K \times \square^q ) \times \mathbb{A}^1 ,
\end{align*}
where the first and the third isomorphisms are the exchanging maps. 

For a moment, fix a positive integer $d > 0$. Consider the $d$-th power  morphism
\[ \rho^d : \mathbb{A}^1 \to  \mathbb{A}^1, \ \ t \mapsto t^d . \]
Note that $\rho^d $ is a finite morphism.
Set
\[
f_q^d := \mathrm{id}_{\mathscr{Y}_K^\circ \times \square^q } \times \rho^d : \mathscr{Y}_K^\circ \times \square^q \times \mathbb{A}^1 \to  \mathscr{Y}_K^\circ \times \square^q \times \mathbb{A}^1 .
\]

For any algebraic cycle $V \subset \mathscr{Y}^\circ \times \square^q $, define another cycle $\mathcal{H}_q^d (V)$ on $\mathscr{Y}_K^\circ \times \square^{q+1} = \mathscr{Y}_K^\circ \times \square^{q} \times \mathbb{A}^1$ by
\[
\mathcal{H}_q^d (V) := (f_q^{d+1} )_\ast  \Psi^\ast (V_K \times \mathbb{A}^1 ) - (f_q^{d} )_\ast  \Psi^\ast (V_K \times \mathbb{A}^1 ),
\]
where $V_K := \pi^\ast V$.
For later use, we set $h_q^d (V) := (f_q^d )_\ast \Psi^\ast (V_K \times \mathbb{A}^1 )$.

\ 

The aim of this subsection is the following theorem.

\begin{thm}\label{construction-of-homotopy-1}
For all integers $q \geq 0$ and $d>0$, the above construction defines a linear map
\[
\mathcal{H}_q^d : \underline{z}^r (\mathscr{Y}^\circ ,q) \to \underline{z}^r (\mathscr{Y}_K^\circ ,q+1)
\]
satisfying the following properties for any $V \in \underline{z}^r (\mathscr{Y}^\circ ,q)$.
\begin{itemize}
\item[$\mathrm{(a)}$]
\[d_{q+1} \mathcal{H}_q^d (V) - \mathcal{H}_{q-1}^d d_q (V) = (-1)^q \left( i_1^\ast \mathcal{H}_q^d (V) - i_0^\ast \mathcal{H}_q^d (V) \right) ,\]
where we set $\mathcal{H}_{-1} = 0$, and 
\[i_\epsilon : \mathscr{Y}_K^\circ \times \square^q \cong \mathscr{Y}_K^\circ \times \square^q \times \{\epsilon \} \hookrightarrow \mathscr{Y}_K^\circ \times \square^{q} \times \mathbb{A}^1  = \mathscr{Y}_K^\circ \times \square^{q+1} .\]
\item[$\mathrm{(b)}$]
\[\mathcal{H}_q^d (\underline{z}^r_w (\mathscr{Y}^\circ ,q)) \subset \underline{z}^r_w (\mathscr{Y}_K^\circ ,q+1).\]
\item[$\mathrm{(c)}$]
\[i_0^\ast \mathcal{H}_q^d (V) = V_K = \pi^\ast V ,\ \ \ i_1^\ast \mathcal{H}_q^d (V) \in \underline{z}^r_w (\mathscr{Y}^\circ_K ,q).\]
\end{itemize}

Moreover, assume that $V$ belongs to the subgroup $\underline{z}^r (\mathscr{Y} ,q)^{(\prime )} \subset \underline{z}^r (\mathscr{Y}^\circ  ,q)$.
Then, there exists a positive integer $d_V$ associated to $V$ such that 
\begin{itemize}
\item[$\mathrm{(d)}$]
\[\mathcal{H}_q^d (V) \in \underline{z}^r (\mathscr{Y}_K ,q+1)^{(\prime )} \]
\end{itemize}
for any $d \geq d_V$.
\end{thm}

We prepare some elementary lemmas.

\begin{lem}\label{elementary-moving} 
Let $k$ be a field, $X$ an equidimensional scheme of finite type over $k$, and $G$ a connected algebraic $k$-group.
Suppose that we are given an action $\gamma : G \times X \to X$.
Let $Z,W$ be two closed subsets of $X$ and let $\gamma |_{G \times Z}$ be the composition $\gamma |_{G \times Z} : G \times Z \hookrightarrow G \times X \xrightarrow{\gamma } X$, where the first map is the natural inclusion.
Assume that the map $\gamma |_{G \times Z}$ is dominant and all the fibers $(\gamma |_{G \times Z} )^{-1} (x) \ (x \in X)$  have the same dimension.
Then, there exists a non-empty open subset $U \subset G$ such that for any point $g \in U$, the closed subsets $g \cdot Z := \gamma (g,Z)$ and $W$ meet properly in $X$.
\end{lem}

\begin{proof}
See \cite[Lemma 1.1]{B}.
 \end{proof}

\begin{lem}\label{Archimedean-property}
Let $X$ be a quasi-compact scheme and let $D, E$ be Cartier divisors on $X$ with $E \geq 0$.
Assume that the restriction of $D$ to the open subset $X \setminus E \subset X$ is effective.
Then, there exists a natural number $n_0 \geq 1$ such that $D + n \cdot E$ is effective for any $n \geq n_0$.
\end{lem}

\begin{proof}
It suffices to prove that there exists at least one number $n \geq 1$ such that $D + n \cdot E \geq 0$.
By quasi-compactness we may assume that $X=\mathrm{Spec}(A)$ is affine and that the Cartier divisors $D$ and $E$ are given by rational functions $\frac{f}{g}$ and $h$ respectively, where $f,g,h \in A$.
The assumption that $D |_{X \setminus E} \geq 0$ means that $\frac{f}{g} \in A_h$, where $A_h$ denotes the localization of $A$ obtained by inverting $h$.
Note that the canonical map $A \to A_h$ is injective because $h$ is a non-zero divisor.
Since every element of $A_h$ is of the form $\frac{a}{h^m}$, where $a \in A$ and $m \geq 1$, we can find a natural number $n \geq 1$ such that $\left( \frac{f}{g}\right) h^n$ is in (the image of) $A$, which implies that $D + n \cdot E$ is effective. This finishes the proof.
 \end{proof}

The rest of this subsection is devoted to the proof of Theorem \ref{construction-of-homotopy-1}.
The proof of \cite[Lemma 2.2]{B} shows that, for any finite set of locally closed subsets $w$ (possibly empty), the linear map
\[
h_q : \underline{z}^r_w (\mathscr{Y}^\circ ,q) \to \underline{z}^r_w (\mathscr{Y}_K^\circ ,q+1), \ \ V \mapsto \Psi^\ast (V_K \times \mathbb{A}^1 )
\]
is well-defined.
Taking $w=\emptyset$, we obtain a map $\underline{z}^r (\mathscr{Y}^\circ ,q) \to \underline{z}^r (\mathscr{Y}_K^\circ ,q+1)$.
It is also shown in {\it loc. cit.} that $i_1^\ast h_q (\underline{z}^r (\mathscr{Y}^\circ ,q) ) \subset \underline{z}^r_w (\mathscr{Y}^\circ_K ,q)$.

We prove that the map
\[
h_q^d : \underline{z}^r_w (\mathscr{Y}^\circ ,q) \to \underline{z}^r_w (\mathscr{Y}_K^\circ ,q+1), \ \ V \mapsto (f_q^d )_\ast \Psi^\ast (V_K \times \mathbb{A}^1 )
\]
is also well-defined and $i_1^\ast h_q^d (\underline{z}^r (\mathscr{Y}^\circ ,q) ) \subset \underline{z}^r_w (\mathscr{Y}^\circ_K ,q)$.
We need the following lemma.

\begin{lem}\label{lem-1}
For a closed point $\tau \in \mathbb{A}^1$, define
\[i_\tau : \mathscr{Y}_K^\circ \times \square^q \cong \mathscr{Y}_K^\circ \times \square^q \times \{\tau \} \hookrightarrow \mathscr{Y}_K^\circ \times \square^{q} \times \mathbb{A}^1  = \mathscr{Y}_K^\circ \times \square^{q+1} .\]
Assume that  $\tau \neq 0$. Then, we have $i_\tau^\ast h_q (\underline{z}^r (\mathscr{Y}^\circ ,q) ) \subset \underline{z}^r_w (\mathscr{Y}^\circ ,q)$. 
\end{lem}
\begin{proof}
First, note that 
\[
\Psi^\ast (Z_K \times \mathbb{A}^1 )|_\tau
=
\psi (K \otimes_k \tau ) \cdot Z_K  \times \{\tau \}.
\]
Here, recall from the assumption (2) in Theorem \ref{moving-lemma} that $\psi (K \otimes_k \tau )$ is a $K'=K \otimes_k \overline{k}$-point of $G_K$, whose image lies over the generic point $\eta $ of $G$.
Applying $\psi (K \otimes_k \tau )^{-1}$, we are reduced to showing that 
$Z_K \times \{\tau \}$ meets
\[
\psi (K \otimes_k \tau )^{-1} \cdot W_K \times F \times \{\tau \}
\subset
(\eta^{-1} \cdot W \times F \times \{\tau \})_K
\]
properly in $\mathscr{Y}^\circ_K \times \mathbb{A}^q \times \{\tau \}$.
Since the field extensions do not change the dimensions of closed subsets,
it suffices to prove that 
$\eta^{-1} \cdot W \times F$ and $Z$ meet properly in $\mathscr{Y}^\circ \times \mathbb{A}^q$.
In other words, we want to prove 
\begin{equation*}
\mathrm{codim}_{\mathscr{Y}^\circ \times \mathbb{A}^q} (Z \cap \eta^{-1} \cdot W \times F) \geq \mathrm{codim}_{\mathscr{Y}^\circ \times \mathbb{A}^q} Z + \mathrm{codim}_{\mathscr{Y}^\circ \times \mathbb{A}^q} (\eta^{-1} \cdot W \times F), 
\end{equation*}
which can be rewritten as 
\begin{equation*}
\dim Z + \dim(\eta^{-1} \cdot W \times F) - \dim(\mathscr{Y}^\circ \times \mathbb{A}^q) \geq \dim(Z \cap \eta^{-1} \cdot W \times F).
\end{equation*}

Since $Z$ and $\mathscr{Y}^\circ \times F$ meet properly in $\mathscr{Y}^\circ \times \mathbb{A}^q$ by assumption, we have the following inequality
\[
\mathrm{codim}_{\mathscr{Y}^\circ \times \mathbb{A}^q } (Z \cap \mathscr{Y}^\circ \times F)
\geq
\mathrm{codim}_{\mathscr{Y}^\circ \times \mathbb{A}^q} Z + \mathrm{codim}_{\mathscr{Y}^\circ \times \mathbb{A}^q} (\mathscr{Y}^\circ \times F),
\]
which can be written as 
\[
\dim Z + \dim(\mathscr{Y}^\circ \times F) - \dim(\mathscr{Y}^\circ \times \mathbb{A}^q) \geq \dim(Z \cap \mathscr{Y}^\circ \times F).
\]
So, to prove the desired inequality, it suffices to show 
\[
\dim (Z \cap \mathscr{Y}^\circ \times F) + \dim(\eta^{-1} \cdot W \times F) - \dim(\mathscr{Y}^\circ \times F) \geq \dim(Z \cap \eta^{-1} \cdot W \times F),
\]
which is equivalent to that $\eta^{-1} \cdot W \times F$
and
$Z \cap \mathscr{Y}^\circ \times F$
meet properly in $\mathscr{Y}^\circ \times F$.
The last assertion follows from Lemma \ref{elementary-moving}, since $\eta^{-1} \in G$ is generic.
 \end{proof}

Consider the closed subset $F \subset \square^q \times \mathbb{A}^1 $ of the form
\[ F=F' \times F'',\]
where $F' \subset \square^q$ is a face 
and $F'' \subset \mathbb{A}^1 $ is an irreducible closed subset.
Then, for any $W \in \{\mathscr{Y}^\circ\} \cup w$, 
\begin{align*}
| h_q^d (V) | \cap |W_K \times F| 
&=^\dag f_q^d (|\Psi^\ast (V_K \times \mathbb{A}^1 )|)   \cap |W_K \times F| \\
&=^{\dag \dag} f_q^d (|\Psi^\ast (V_K \times \mathbb{A}^1 )| \cap (f_q^d )^{-1} (|W_K \times F|) ) \\
&=^{\dag \dag \dag } f_q^d (|\Psi^\ast (V_K \times \mathbb{A}^1 )| \cap |W_K \times F' \times (\rho^d )^{-1} (F'')| ),
\end{align*}
where $=^\dag$ holds by definition of $h^{(d)}_q$. $=^{\dag \dag}$ is just the set-theoretical projection formula. Finally, $=^{\dag \dag \dag}$ is trivial. 
Since $f_q^d$ is finite, the intersection $| \underline{h}^{(d)}_q (Z) | \cap |W_K \times F|$ has the right dimension if and only if $|\Psi^\ast (Z_K \times \mathbb{A}^1 )| \cap |W_K \times F' \times (\rho^d )^{-1} (F'')|$ does.
If $F''=\{0\}$ or $\mathbb{A}^1$, then we have $|F' \times (\rho^d )^{-1} (F'')| = |F|$.
If $F''=\{\tau \}$ for a closed point $\tau \in \mathbb{A}^1 - \{0\}$,
then $|(\rho^d )^{-1} (F'')|=|(\rho^d )^{-1} (\{\tau \})|$ consists of finitely many closed points of $\mathbb{A}^1 - \{0\}$.
Thus, Lemma \ref{lem-1} shows that the map
\[
h_q^d : \underline{z}^r_w (\mathscr{Y}^\circ ,q) \to \underline{z}^r_w (\mathscr{Y}_K^\circ ,q+1), \ \ V \mapsto (f_q^d )_\ast \Psi^\ast (V_K \times \mathbb{A}^1 )
\]
is well-defined and $i_1^\ast h_q^d (\underline{z}^r (\mathscr{Y}^\circ ,q) ) \subset \underline{z}^r_w (\mathscr{Y}^\circ_K ,q)$.
In particular, we obtain a well-defined map $\mathcal{H}_q^d := h_q^{d+1} - h_q^{d}$ satisfying (b).

\ 

To prove the assertion (a), note that for each $i=1,\dots ,q-1$ and $\epsilon = 0,1$, we have $d_i^\epsilon (f_{q}^d )_\ast = (f_{q-1}^d )_\ast d_i^\epsilon$ and $d_i^\epsilon \Psi^\ast = \Psi^\ast d_i^\epsilon$, where the first equality follows from the cartesian diagram
\[\begin{xymatrix}{
\square^{q-1} \times \mathbb{A}^1 \ar[r]^{\delta_i^\epsilon \times \mathbb{A}^1 } \ar[d]_{f_{q-1}^d } \ar@{}[rd]|\square & \square^q \times \mathbb{A}^1 \ar[d]^{f_q^d } \\
\square^{q-1} \times \mathbb{A}^1 \ar[r]^{\delta_i^\epsilon \times \mathbb{A}^1 } & \square^q \times \mathbb{A}^1
}\end{xymatrix}\]
and the projection formula. 
The second equality is obvious.
This proves the assertion (a).

\ 

The assertion (c) is a consequence of the projection formula
\[
(f_q^d )_\ast (\Psi^\ast (Z_K \times \mathbb{A}^1 )) \cdot \mathscr{Y}_K^\circ \times \square^q \times \{0\} = (f_q^d )_\ast (\Psi^\ast (Z_K \times \mathbb{A}^1 ) \cdot (f_q^d )^\ast (\mathscr{Y}_K^\circ \times \square^q \times \{0\})),
\]
where $\cdot$ denotes the intersection product, combined with the equations
\begin{align*}
(f_q^d )^\ast (\mathscr{Y}_K^\circ \times \square^q \times \{0\}) = d \cdot \mathscr{Y}^\circ_K \times \square^q \times \{0\}, \ \ 
f_q^d |_{\mathscr{Y}_K^\circ \times \square^q \times \{0\}} = \mathrm{id}_{\mathscr{Y}_K^\circ \times \square^q \times \{0\}}.
\end{align*}

\ 

Finally, we prove the assertion (d).
Assume that $V \in \underline{z}^r (\mathscr{Y},q)$.
We must show that $h_q^d (V)$ belongs to $\underline{z}^r (\mathscr{Y}_K ,q+1)$ for $d \gg 0$.
Since $\underline{z}^r (\mathscr{Y} ,q)$ is a subgroup of $\underline{z}^r (\mathscr{Y}^\circ ,q)$, we have $h_q^d (V) \in \underline{z}^r (\mathscr{Y}^\circ_K ,q+1)$ for any $d>0$.
Therefore, it suffices to check the modulus condition.

We may assume that $V$ is an integral algebraic cycle.
Let $\overline{V}_K \subset \overline{Y}_K \times (\mathbb{P}^1 )^q$ be the closure of $V_K$, and let $\overline{V}_K^N \to \overline{V}_K$ be its normalization.
Since $V \in \underline{z}^r (\mathscr{Y},q)$, we have $V_K \in \underline{z}^r (\mathscr{Y}_K ,q)$, hence
\[
Y^\infty_K \times (\mathbb{P}^1 )^q  |_{\overline{V}_K^N }
\leq \overline{Y}_K \times F_q   |_{\overline{V}_K^N } ,
\]
where $(-)|_{\overline{V}^N }$ denotes the restriction of Cartier divisors to $\overline{V}_K^N$.
Let $\overline{V_K \times \mathbb{A}^1}$ denote the closure of $V_K \times \mathbb{A}^1$ in $\overline{Y}_K \times (\mathbb{P}^1 )^q \times \mathbb{A}^1$.
Then, we have $\overline{V_K \times \mathbb{A}^1} = \overline{V}_K \times \mathbb{A}^1$.
So, the above inequality implies
\[
Y^\infty_K \times (\mathbb{P}^1 )^q \times \mathbb{A}^1 |_{\overline{V_K \times \mathbb{A}^1 }^N }
\leq \overline{Y}_K \times F_q \times \mathbb{A}^1  |_{\overline{V_K \times \mathbb{A}^1 }^N } .
\]

Note that $\Psi$ extends to an isomorphism
\[\overline{\Psi} : \overline{Y}_K \times (\mathbb{P}^1 )^q \times \mathbb{A}^1 \to \overline{Y}_K \times (\mathbb{P}^1 )^q \times \mathbb{A}^1  \]
such that $\overline{\Psi}^\ast (Y_K^\infty \times \mathbb{A}^1) = Y_K^\infty \times \mathbb{A}^1$.
Let $\overline{h_q (V)}$ be the closure of $h_q (V) = \Psi^\ast (V_K \times \mathbb{A}^1 )$ in $\overline{Y}_K \times (\mathbb{P}^1 )^q \times \mathbb{A}^1$, and let $\overline{h_q (V)}^N$ be its normalization. 
Then, we have $\overline{h_q (V)} = \overline{\Psi }^\ast (\overline{V_K \times \mathbb{A}^1 })$.
Therefore, we have
\[
Y^\infty_K \times (\mathbb{P}^1 )^q \times \mathbb{A}^1 |_{\overline{h_q (V) }^N }
\leq \overline{Y}_K \times F_q \times \mathbb{A}^1  |_{\overline{h_q (V) }^N }.
\]
Moreover, let $\overline{\overline{h_q (V) }}$ be the closure of $h_q (V)$ in $\overline{Y}_K \times (\mathbb{P}^1 )^q \times \mathbb{P}^1$, and let $\overline{\overline{h_q (V) }}^N$ be its normalization.
Since
\[\overline{Y}_K \times (\mathbb{P}^1 )^q \times \mathbb{A}^1 = \overline{Y}_K \times (\mathbb{P}^1 )^q \times \mathbb{P}^1 - \overline{Y}_K \times (\mathbb{P}^1 )^q \times \{\infty \},  \]
Lemma \ref{Archimedean-property} implies that there exists a positive integer $d_V$ such that 
\[
Y^\infty_K \times (\mathbb{P}^1 )^q \times \mathbb{P}^1 |_{\overline{\overline{h_q (V) }}^N }
\leq \overline{Y}_K \times F_q \times \mathbb{P}^1  |_{\overline{\overline{h_q (V) }}^N } + d \cdot \overline{Y}_K \times (\mathbb{P}^1 )^q \times \{\infty \}  |_{\overline{\overline{h_q (V) }}^N }
\]
for any $d \geq d_V$.

Note that the $d$-th power map $\rho^d : \mathbb{A}^1 \to \mathbb{A}^1 , t \mapsto t^d$ naturally extends to a finite morphism
\[
\overline{\rho }^d : \mathbb{P}^1 \to \mathbb{P}^1 , \ \ [T_0 : T_1] \mapsto [T_0^d : T_1^d ], \hspace{5.0mm} t = \frac{T_1 }{T_0 } .
\]
Let $\overline{h_q^d (V)}$ be the closure of $h_q^d (V)$ in $\overline{Y}_K \times (\mathbb{P}^1 )^q \times \mathbb{P}^1$, and let $\overline{h_q^d (V)}^N$ be its normalization. 
Then, we have a commutative diagram of finite surjective morphisms
\[\begin{xymatrix}{
\overline{\overline{h_q (V) }}^N \ar@{>>}[r] \ar@{>>}[d] & \overline{h_q^d (V)}^N \ar@{>>}[d] \\
\overline{\overline{h_q (V) }} \ar@{>>}[r]^{\mathrm{id} \times \overline{\rho }^d }  & \overline{h_q^d (V)} ,
}\end{xymatrix}\]
where the vertical maps are the normalization morphisms, and the bottom horizontal arrow is the restriction of  $\mathrm{id} \times \overline{\rho }^d $, where $\mathrm{id}$ is the identity map on $\overline{Y} \times (\mathbb{P}^1 )^q$.
The top horizontal arrow is induced by the universal property of the normalization.
Since $(\overline{\rho}^d )^\ast \{\infty \} = d \cdot \{\infty \}$, Lemma \ref{effectivity-criterion} shows that the last inequality implies 
\[
Y^\infty_K \times (\mathbb{P}^1 )^q \times \mathbb{P}^1 |_{\overline{h_q^d (V) }^N }
\leq \overline{Y}_K \times F_q \times \mathbb{P}^1  |_{\overline{h_q^d (V) }^N } +  \overline{Y}_K \times (\mathbb{P}^1 )^q \times \{\infty \}  |_{\overline{h_q^d (V) }^N } .
\]
This finishes the proof of Theorem \ref{construction-of-homotopy-1} for $\underline{z}^r (\mathscr{Y},q)$.
The proof for the na\"{i}ve version $\underline{z}^r (\mathscr{Y},q)'$ is obtained by replacing every $F_q$ in the above proof with the trivial divisor.

\subsection{End of proof of the moving lemma}

We are now ready to prove the assertion (2) of Proposition \ref{key-prop-1-2}.

\begin{prop}
Let $C_\ast \subset \underline{z}^r (\mathscr{Y},\ast )^{(\prime )} $ be a finitely generated subcomplex, and set $C_{w,\ast } := C_\ast \cap \underline{z}^r_w (\mathscr{Y},\ast )^{(\prime )} $.
Then, the composite map
\[
\overline{\pi}^\ast_C : C_\ast / C_{w,\ast } \hookrightarrow \underline{z}^r (\mathscr{Y},\ast )^{(\prime )} / \underline{z}^r_w (\mathscr{Y},\ast )^{(\prime )}  \xrightarrow{\overline{\pi }} \underline{z}^r (\mathscr{Y}_K ,\ast )^{(\prime )} / \underline{z}^r_w (\mathscr{Y}_K ,\ast )^{(\prime )}
\]
is homotopic to zero.
In particular, by taking the inductive limit over $C_\ast$, we can see that the assertion (2) of Proposition \ref{key-prop-1-2} holds.
\end{prop}

\begin{proof}
Fix finite number of irreducible closed subsets \[V_l^q \subset \mathscr{Y}^\circ \times \mathbb{A}^q\] such that $V_l^q \in \underline{z}^r (\mathscr{Y},q)^{(\prime )}$ and the family $\{V_\ast^{q-1}\}$ contains all components of all cycles $d_q (V_l^q)$.
Let $C_q$ be the free abelian group generated by $\{V_l^q\}_l$.
Then, we obtain a finitely generated subcomplex \[C_\ast \subset \underline{z}^r (\mathscr{Y},q)^{(\prime )} .\]
Since the complexes of this type are cofinal among all finitely generated subcomplexes, we may assume that the subcomplex $C_\ast$ in the assertion of the proposition is of this type.

For each $V_l^q $, take a positive integer $d_{V_l^q }$ as in Theorem \ref{construction-of-homotopy-1}.
Set 
\[
d:=\max_{l,q} d_{V_l^q}.
\]
Then, by the properties (b) and (d) in Theorem \ref{construction-of-homotopy-1}, we obtain a linear map
\[
\mathcal{H}_q^d (-) : C_q / C_{w,q} \to \underline{z}^r (\mathscr{Y}_K ,q+1 )^{(\prime )} / \underline{z}^r_w (\mathscr{Y}_K ,q+1 )^{(\prime )}
\]
for each $q \geq 0$ (recall that $\underline{z}^r_w (\mathscr{Y}_K ,q+1 )^{(\prime )} = \underline{z}^r(\mathscr{Y}_K ,q+1 )^{(\prime )} \cap \underline{z}^r_w (\mathscr{Y}^\circ_K ,q+1 )$).
The properties (a) and (c) shows that this is a homotopy between $\overline{\pi}^\ast_C$ and zero, noting that $i_1^\ast \mathcal{H}_q^d (V) \in \underline{z}^r (\mathscr{Y}_K ,q)^{(\prime )} \cap \underline{z}^r_w (\mathscr{Y}^\circ_K ,q)^{(\prime )} = \underline{z}^r_w (\mathscr{Y}_K ,q)^{(\prime )}$ for any $V \in C_q$.
This finishes the proof.
 \end{proof}

Thus, we finished the proof of Theorem \ref{moving-lemma}.

\subsection{Proof of cube invariance}\label{pf-of-inv}
Now, we are ready to provide a proof of Theorem \ref{cube-invariance}, following the method of \cite[Corollary 2.6]{B}. 
For simplicity of notations, we treat the higher Chow group with modulus.
To obtain the proof for the {\it na\"{i}ve} higher Chow group with modulus, replace every $z^r_{(w)} (-,\ast ,A), \mathrm{CH}^r (-,\ast ,A)$ with $z^r_{(w)} (-,\ast ,A)', \mathrm{CH}^r (-,\ast ,A)'$, respectively.

Our goal is to prove that the pull-back map
\[
\mathrm{pr}_1^\ast : z^r (\mathscr{X},\ast ,A) \to z^r (\mathscr{X} \otimes \overline{\square}^{(-1)},A) 
\]
induces an isomorphism on homology groups
(recall that $A$ is the $\mathbb{Z}$-module we fixed at the beginning of this section).
First, we prove injectivity.
Set $w:=\Set{\mathscr{X}^\circ \times \{0\}}$ and let $i_0$ be a minimal morphism defined by
\[
i_0 : \mathscr{X}  \cong \mathscr{X}  \otimes \{0 \} \hookrightarrow \mathscr{X} \otimes \overline{\square}^{(-1)},
\]
where the latter map is the natural inclusion.
Consider the composite
\[
z^r (\mathscr{X},\ast ,A) \xrightarrow{\mathrm{pr}_1^\ast } z^r_w (\mathscr{X} \otimes \overline{\square}^{(-1)},\ast ,A)  \xrightarrow{i_0^\ast } z^r (\mathscr{X},\ast ,A),
\]
where $i_0^\ast $ is well-defined by definition of $w$.
We can also check that $\mathrm{pr}_1^\ast $ is well-defined. Indeed, fix $q \geq 0$ and take any prime cycle $V \in z^r (\mathscr{X},q ,A)$. It suffices to prove that $\mathrm{pr}_1^\ast V=V \times \mathbb{A}^1 \subset \mathscr{X}^\circ \times \mathbb{A}^1 \times \mathbb{A}^q$ meets $W \times F$ properly for any $W \in \Set{\mathscr{X}^\circ \times \{0\}, \mathscr{X}^\circ \times \mathbb{A}^1 }$ and for any face $F \subset \mathbb{A}^q$.
If $W=\mathscr{X}^\circ \times \{0\}$, then $V \times \mathbb{A}^1 \cap W \times F = (V \cap \mathscr{X}^\circ \times F) \times \{0\}$ has a proper dimension.
If $W=\mathscr{X}^\circ \times \mathbb{A}^1 $, then $V \times \mathbb{A}^1 \cap W \times F = (V \cap \mathscr{X}^\circ \times F) \times \mathbb{A}^1 $ has a proper dimension.
Thus, the well-defined maps induce the following maps on homology groups for any $q \geq 0$:
\[
\mathrm{CH}^r (\mathscr{X},q ,A) \xrightarrow{\mathrm{pr}_1^\ast } \mathrm{H}_q (z^r (\mathscr{X} \otimes \overline{\square}^{(-1)},\ast ,A)) \xrightarrow{i_0^\ast } \mathrm{CH}^r (\mathscr{X},q ,A),
\]
where the middle group is canonically isomorphic to $\mathrm{CH}^r (\mathscr{X} \otimes \overline{\square}^{(-1)},q ,A)$ by Corollary \ref{special-moving-lemma}.
It is obvious that the composite is the identity map, which implies that $\mathrm{pr}_1^\ast $ is an injection, as desired.

We prove surjectivity of $\mathrm{pr}_1^\ast $ on homology.
Consider the multiplication map
\[
\mu : \mathbb{A}^1 \times \mathbb{A}^1 \to \mathbb{A}^1 , \ \ (s,t) \mapsto st.
\]
In fact, we have the following lemma, which is proven in \cite[Lemma 6.1]{KSY} by using a blowing-up technique.
\begin{lem}\label{adm-multiplication}
The morphism $\mu$ of schemes defines a flat, left fine, {\it admissible} morphism of modulus pairs
\[
\mu : \overline{\square} \otimes \overline{\square} \to \overline{\square}.
\]
\end{lem}
By Lemma \ref{adm-multiplication} and by Remark \ref{adm-coadm-minus}, $\mu$ gives a flat, left fine and {\it coadmissible} map \[\overline{\square}^{(-1)} \otimes \overline{\square}^{(-1)} \to \overline{\square}^{(-1)}\] since $\overline{\square}^{(-1)} \otimes \overline{\square}^{(-1)} = (\overline{\square} \otimes \overline{\square})^{(-1)}.$
Thus, by Proposition \ref{functoriality} (1), we obtain a pullback map \[\mu^\ast : z^r (\mathscr{X} \otimes \overline{\square}^{(-1)},\ast) \to z^r(\mathscr{X} \otimes \overline{\square}^{(-1)} \otimes \overline{\square}^{(-1)},\ast). \]

By Corollary \ref{special-moving-lemma} again, we obtain the following composite map:
\begin{align*}
&H_q (z^r_w (\mathscr{X} \otimes \overline{\square}^{(-1)} ,\ast ,A)) \\
&\xrightarrow{\cong} \mathrm{CH}^r (\mathscr{X} \otimes \overline{\square}^{(-1)} ,q,A) 
\xrightarrow{\mu^\ast } \mathrm{CH}^r (\mathscr{X} \otimes \overline{\square}^{(-1)} \otimes \overline{\square}^{(-1)} ,q,A) \\
&\xrightarrow{\cong}
H_q (z^r_{w'} (\mathscr{X} \otimes \overline{\square}^{(-1)} \otimes \overline{\square}^{(-1)} ,\ast ,A)) \xrightarrow{i_\epsilon^{'\ast} } \mathrm{CH}^r (\mathscr{X} \otimes \overline{\square}^{(-1)} ,q,A).
\end{align*}
Here, $w':=\{ \mathscr{X}^\circ \times \mathbb{A}^1 \times \{0,1\} \}$, $\epsilon \in \{0,1\}$ and $i'_\epsilon$ is a minimal morphism defined by
\[
i'_\epsilon : \mathscr{X} \otimes \overline{\square}^{(-1)} \cong \mathscr{X} \otimes \overline{\square}^{(-1)} \otimes \{\epsilon \} \hookrightarrow \mathscr{X} \otimes \overline{\square}^{(-1)} \otimes \overline{\square}^{(-1)} ,
\]
where the latter map is the natural inclusion.
Take any cycle $Z \in z^r_w (\mathscr{X} \otimes \overline{\square}^{(-1)},q)$.
Since $i_0^{'\ast } $ and $i_1^{'\ast } $ coincide on homology by Lemma \ref{rigidity}, we have
\[
i_0^{'\ast } \mu^\ast (Z) = i_1^{'\ast } \mu^\ast  (Z)
\]
on homology.
The left hand side is equal to $\mathrm{pr}_1^\ast i_0^\ast (Z)$ since $\mu i'_0 = i_0 \mathrm{pr}_1$, and the right hand side is equal to $Z$, which concludes the proof of Theorem \ref{cube-invariance}.

\section{Nilpotent higher Chow group and a module structure over $\mathbb{W}(k)$}



We define a cycle-theoretic analogue of nilpotent $K$-group.
We fix a base field $k$ and an abelian group $A$.

\begin{dfn}
For a modulus pair $\mathscr{X}=(\overline{X},X^\infty )$ and for any integer $r \geq 0$, define
\[
\mathrm{N}z^r (\mathscr{X},\ast ,A):=\mathrm{Coker}\left[ \mathrm{pr}_1^\ast : z^r (\mathscr{X},\ast ,A) \to z^r (\mathscr{X}\otimes \mathbb{A}^1 ,\ast ,A) \right].
\]
We call the homology group
\[
\mathrm{NCH}^r (\mathscr{X},q,A):=\mathrm{H}_q (\mathrm{N}z^r (\mathscr{X},\ast ,A))
\]
the {\it nilpotent higher Chow group with modulus}.
\end{dfn}

The following proposition shows that the nilpotent higher Chow group is the obstruction for the $\mathbb{A}^1$-homotopy invariance.

\begin{prop}\label{nilp-split}
For any modulus pair $\mathscr{X}$ and for any $r,q \geq 0$, we have a canonical splitting
\[
\mathrm{CH}^r (\mathscr{X} \otimes \mathbb{A}^1 ,q,A) \cong \mathrm{CH}^r (\mathscr{X},q,A) \oplus \mathrm{NCH}^r (\mathscr{X},q,A).
\]
\end{prop}

\begin{proof}
By Corollary \ref{special-moving-lemma}, the minimal morphism $i_0 : \mathscr{X} \cong \mathscr{X} \otimes \{0\} \to \mathscr{X} \otimes \mathbb{A}^1$ induces a map $i_0^\ast : \mathrm{CH}^r (\mathscr{X} \otimes \mathbb{A}^1 ,q,A) \to \mathrm{CH}^r (\mathscr{X},q,A)$.
This is a retraction of the map $\mathrm{pr}_1^\ast : \mathrm{CH}^r (\mathscr{X},q,A) \to \mathrm{CH}^r (\mathscr{X} \otimes \mathbb{A}^1 ,q,A)$.
This finishes the proof.
 \end{proof}

\subsection{Module structure over the ring of Witt vectors}

Recall that the big Witt ring $\mathbb{W}(k)$ is defined as an abelian group 
\[
\mathbb{W}(k)=\left( 1 + uk\llbracket u \rrbracket \right)^\times 
\]
with a unique continuous functorial multiplication $\star$ satisfying $(1-au) \star (1-bu) = 1-abu$.
For any integer $m \geq 0$, the truncated big Witt ring $\mathbb{W}_m (k)$ is defined as the quotient
\[
\mathbb{W}_m (k):=\mathbb{W}(k) / \left( 1 + u^{m+1} k\llbracket u \rrbracket \right)^\times \cong (1+u k[u]/u^{m+1} k[u] )^\times .
\]
Note that every element of $\mathbb{W}_m (k)$ is represented by a polynomial. Moreover, we have
\[
\mathbb{W}(k) \xrightarrow{\cong } \varprojlim_{m} \mathbb{W}_m (k).
\]
We regard $\mathbb{W}(k)$ as a topological ring with respect to the projective limit topology, with each $\mathbb{W}_m (k)$ given the discrete topology.

The goal of this section is to prove the following result.
Fix a $\mathbb{Z}$-module $A$.

\begin{thm}\label{Witt-action}
For any modulus pair $\mathscr{X}$ over $k$ and for arbitrary integers $r,q\geq 0$, the nilpotent higher Chow group $\mathrm{NCH}^r (\mathscr{X},q,A)$ has a structure of $\mathbb{W}(k)$-module, where $\mathbb{W}(k)$ denotes the big Witt-ring of $k$, which is continuous in the following sense:
\begin{itemize}
\item[]
For any element $V \in \mathrm{NCH}^r (\mathscr{X},q,A)$, there exists a positive integer $m \geq 1$ such that $(1+u^m k\llbracket u \rrbracket ) \cdot V = 0$ in $\mathrm{NCH}^r (\mathscr{X},q,A)$.
\end{itemize}
\end{thm}

Before starting the proof, we recall the following result.
It is not used in the proof of the above theorem, but it motivates our construction of the action.

\begin{thm}
For any $m \geq 1$, we have a canonical isomorphism of rings 
\[
\mathrm{div} : \mathbb{W}_m (k) \xrightarrow{\cong } \mathrm{CH}^1 (\mathbb{A}^1 |m\{0\}).
\]
\end{thm}

\begin{proof}
The reader can find a proof in \cite{Ru}.
We recall only the construction of $\mathrm{div}$.
Let $f(u) = 1+u g(u) \in 1+u k[u]$ be a polynomial.
Then, it defines a divisor $\mathrm{div}(f)$ on $\mathbb{A}^1 = \mathbb{A}^1_k$.
 \end{proof}

This result leads us to the following construction.
Consider the natural surjection 
\[
z_0 (\mathbb{A}^1 \setminus \{0\},0) \twoheadrightarrow \mathrm{CH}_0 (\mathbb{A}^1 | m\{0\}),
\]
where the left hand side denotes the free abelian group generated by zero cycles on $\mathbb{A}^1 \setminus \{0\}$.
Note that by definition we have
\[
\mathrm{NCH}^r (\mathscr{X},q,A) = \mathrm{H}_q \left( \frac{z^r (\mathscr{X} \otimes \mathbb{A}^1 ,\ast ,A)}{\mathrm{pr}_1^\ast z^r (\mathscr{X} ,\ast ,A)} \right).
\]

Let $\alpha \in z_0 (\mathbb{A}^1 \setminus \{0\},0)$ be a zero cycle and let $V \in z^r (\mathscr{X} \otimes \mathbb{A}^1 ,q)$ be an integral cycle.
Then we obtain a cycle $\alpha \times V$ on $\mathbb{A}^1 \times (\mathscr{X}^\circ \times \mathbb{A}^1 \times \mathbb{A}^q )$.
Consider the multiplication map
\[
\mu : \mathbb{A}^1 \times \mathbb{A}^1 \to \mathbb{A}^1 , (t,u) \mapsto tu .
\]
Let $|\alpha |$ denote the support of the cycle $\alpha$, which is a finite set of closed points on $\mathbb{A}^1 \setminus \{0\}$.
Then, $\alpha \times V$ can be regarded as a cycle on $|\alpha | \times (\mathscr{X}^\circ \times \mathbb{A}^1 \times \mathbb{A}^q )$.
Since $|\alpha|$ does not contain $0$,  the restriction 
\[\mu_\alpha : |\alpha | \times \mathbb{A}^1 \to \mathbb{A}^1 \] of $\mu $
 is a finite morphism.
We define a cycle $\alpha \cdot V$ on $\mathscr{X}^\circ \times \mathbb{A}^1 \times \mathbb{A}^q $ by
\[
\alpha \cdot V := (\mu_\alpha \times \mathrm{id}_{\mathscr{X}^\circ \times \mathbb{A}^q } )_\ast (\alpha \times V).
\]
This construction is linearly extended to the following homomorphism:
\[
z_0 (\mathbb{A}^1 \setminus \{0\},0) \otimes z^r (\mathscr{X} \otimes \mathbb{A}^1 ,q,A) \to z^r (\mathscr{X} \otimes \mathbb{A}^1 ,q,A) ,\ \  \alpha \otimes V \mapsto \alpha \cdot V,
\]
satisfying $\mathrm{div}(1-u) \cdot V =\{1\} \cdot V=V$.

\begin{prop}
For any $q \geq 0$,
the homomorphism defined as above induces a homomorphism
\[
z_0 (\mathbb{A}^1 \setminus \{0\},0) \otimes \frac{z^r (\mathscr{X} \otimes \mathbb{A}^1 ,q,A)}{\mathrm{pr}_1^\ast z^r (\mathscr{X},q,A)} \to \frac{z^r (\mathscr{X} \otimes \mathbb{A}^1 ,q,A)}{\mathrm{pr}_1^\ast z^r (\mathscr{X},q,A)}.
\]
Moreover, it is compatible with the differential homomorphisms of the complex $\frac{z^r (\mathscr{X} \otimes \mathbb{A}^1 ,\ast ,A)}{\mathrm{pr}_1^\ast z^r (\mathscr{X},\ast  ,A)}$.
\end{prop}

\begin{proof}
The assertions for the $A$-coefficient-case follow from the $\mathbb{Z}$-coefficient-case by applying $\otimes_\mathbb{Z} A$, so we may assume that $A=\mathbb{Z}$.
Let $V \in z^r (\mathscr{X},q)$ be an integral cycle, and let $\alpha \in \mathbb{A}^1 \setminus \{0\}$ be a closed point.
 Then, we have
\begin{align*}
\alpha \cdot (\mathrm{pr}_1^\ast V ) 
&=\alpha \cdot (V \times \mathbb{A}^1 ) 
=(\mu_\alpha \times \mathrm{id}_{\mathscr{X}^\circ \times \mathbb{A}^q })_\ast (\alpha \times V \times \mathbb{A}^1 )\\
&=(\mu_\alpha )_\ast (\alpha \times \mathbb{A}^1 ) \times V 
=(\deg (\alpha ) \cdot \mathbb{A}^1 ) \times V \\
&= \deg (\alpha ) \cdot (\mathrm{pr}_1^\ast V ) .
\end{align*}
In particular, we have $\alpha \cdot \mathrm{pr}_1^\ast z^r (\mathscr{X},q) \subset \mathrm{pr}_1^\ast z^r (\mathscr{X},q)$, which proves the first statement of the lemma.
To prove the compatibility, 
let $V \in z^r (\mathscr{X} \otimes \mathbb{A}^1 ,q)$ be an integral cycle, and let $\alpha \in \mathbb{A}^1 \setminus \{0\}$ be a closed point.
We would like to prove that $d_q (\alpha \cdot V) = \alpha \cdot d_q V$, where
\[
d_q =\sum_{i,\epsilon } (\iota_i^\epsilon )^\ast  : z^r (\mathscr{X} \otimes \mathbb{A}^1 ,q) \to z^r (\mathscr{X} \otimes \mathbb{A}^1 ,q-1)
\]
denotes the $q$-th differential homomorphism.
It is a consequence of the following calculation:
\begin{align*}
(\iota_i^\epsilon )^\ast (\alpha \cdot V)
 &= (\iota_i^\epsilon )^\ast (\mu_\alpha \times \mathrm{id}_{\mathscr{X}^\circ \times \mathbb{A}^q } )_\ast (\alpha \times V) \\
 &= (\mu_\alpha \times \mathrm{id}_{\mathscr{X}^\circ \times \mathbb{A}^{q-1} } )_\ast (\iota_i^\epsilon )^\ast (\alpha \times V) \\
 &=(\mu_\alpha \times \mathrm{id}_{\mathscr{X}^\circ \times \mathbb{A}^{q-1} } )_\ast  (\alpha \times (\iota_i^\epsilon )^\ast V) \\
 &=\alpha \cdot (\iota_i^\epsilon )^\ast V.
\end{align*}
This finishes the proof.
 \end{proof}

By the above proposition, we obtain an action 
\[
z_0 (\mathbb{A}^1 \setminus \{0\},0) \otimes \mathrm{NCH}^r (\mathscr{X},q,A) \to  \mathrm{NCH}^r (\mathscr{X},q,A)
\]
for any $q \geq 0$.
For any polynomial $f(u) \in 1+uk[u]$, the divisor $\mathrm{div} (f)$ belongs to $z_0 (\mathbb{A}^1 \setminus \{0\},0)$.
The following is a key proposition.

\begin{prop}\label{vanishing}
For any $V \in \mathrm{NCH}^r  (\mathscr{X} ,q,A)$, there exists a positive integer $m$ such that $\mathrm{div} (1+u^m k[u]) \cdot V=0$.
\end{prop}

Before starting the proof, we prepare an elementary lemma.

\begin{lem}\label{colimit-lemma}
Let $\mathscr{Y}=(\overline{Y},Y^\infty )$ be a modulus pair over $k$, and let $E$ be an effective modulus pair on $\overline{Y}$ such that $|Y^\infty |$ and $|D|$ have no common components. Let $w$ be a finite set of locally closed subsets of $\overline{Y}\setminus (|E|\cup |Y^\infty |)$.
Then, for any integers $r,q$, we have 
\[
z^r_w ((\overline{Y}\setminus |E|,  Y^\infty |_{\overline{Y}\setminus |E|}) ,q,A) = \bigcup_{m \geq 1} z^r_w ((\overline{Y},Y^\infty - mE), q, A).
\]
\end{lem}
\begin{proof}
The right hand side is a subgroup of the left hand side by definition.
Take any irreducible closed subset $V$ which belongs to the left hand side.
Let $\overline{V}$ be the closure of $V$ in $\overline{Y} \times (\mathbb{P}^1 )^q$, and set $Z=\overline{V} \cap (\overline{Y} \setminus |E|) \times (\mathbb{P}^1 )^q$.
Let $\overline{V}^N$ and $Z^N$ be the normalizations of $\overline{V}$ and $Z$, respectively.
Then, by assumption,
we have 
\[
Y^\infty \times (\mathbb{P}^1 )^q |_{Z^N } \leq \overline{Y} \times F_q |_{Z^N} .
\]
By Lemma \ref{Archimedean-property}, there exists a positive integer $m$ such that
\[
Y^\infty \times (\mathbb{P}^1 )^q |_{\overline{V}^N } \leq \overline{Y} \times F_q |_{\overline{V}^N } + m E \times (\mathbb{P}^1 )^q |_{\overline{V}^N }.
\]
This inequality implies that $V$ belongs to $z^r_w ((\overline{Y},Y^\infty - mE), q, A)$, finishing the proof.
 \end{proof}

\begin{proof}[of Proposition \ref{vanishing}]
Note that $\mathrm{NCH}^r (\mathscr{X},q,A) \cong \mathrm{H}_q \left( \frac{z^r_w (\mathscr{X} \otimes \mathbb{A}^1 ,\ast ,A)}{\mathrm{pr}_1^\ast z^r (\mathscr{X},\ast ,A)} \right)$, where $w=\{\mathscr{X}^\circ \times \{0\}\}$.
Take any element $Z \in  z^r_w (\mathscr{X}\otimes \mathbb{A}^1 ,q,A)$, and fix an irreducible component $V$ of $Z$.
Pulling back the cycle $V \in z^r_w (\mathscr{X}\otimes \mathbb{A}^1 ,q,A)$ by the flat morphism $\mathrm{id}_{\mathscr{X^\circ }} \times \mu$, we obtain a cycle \[V_1 := (\mathrm{id}_{\mathscr{X^\circ }} \times \mu)^\ast V \in z^r_w (\mathscr{X}\otimes \mathbb{A}^1 \otimes \mathbb{A}^1 ,q,A).\]
By the isomorphism $\lambda : \mathbb{A}^1 \cong \mathbb{P}^1 \setminus \{0\}, t \mapsto t^{-1}$, we obtain a cycle \[V_2 := \lambda^\ast V_1 \in z^r_w (\mathscr{X}\otimes \mathbb{A}^1 \otimes (\mathbb{P}^1 \setminus \{0\}) ,q,A) = \bigcup_{m \geq 1} z^r_w (\mathscr{X}\otimes \mathbb{A}^1 \otimes (\mathbb{P}^1 -m\{0\}) ,q,A) \]
corresponding to $V_1$, where the second identity follows from Lemma \ref{Archimedean-property} applied to $\mathscr{Y}=\mathscr{X}\otimes \mathbb{P}^1$ and $E=\overline{X} \times \{0\}$.
Take a positive integer $m$ satisfying $V_2  \in z^r_w (\mathscr{X}\otimes \mathbb{A}^1 \otimes (\mathbb{P}^1 -m\{0\}) ,q,A)$.

Let $f(u) = 1+ u^m g(u) \in 1+u^m k[u]$ be a polynomial.
Then, it defines a finite morphism $f : \mathbb{A}^1 \to \mathbb{A}^1$ which extends to a finite morphism 
\[
F : \mathbb{P}^1 \to \mathbb{P}^1 , \ \ [U_0 : U_1] \mapsto [U_0^{\deg f} : U_0^{\deg f} \cdot f(U_1 / U_0 )].
\]
Then, $F^\ast \{1\}$ is an effective Cartier divisor on $\mathbb{P}^1$ such that $F^\ast \{1\} \geq m\{0\}$.
Therefore, we obtain a composite
\begin{align*}
z^r_w (\mathscr{X} \otimes \mathbb{A}^1 \otimes (\mathbb{P}^1 , -m\cdot \{0\}),q,A) &\xrightarrow{j^\ast } z^r_w (\mathscr{X} \otimes \mathbb{A}^1 \otimes (\mathbb{P}^1 , -F^\ast \{1\}),q,A)  \\
&\xrightarrow{F_\ast } z^r_w (\mathscr{X} \otimes \mathbb{A}^1 \otimes (\mathbb{P}^1 , -\{1\}),q,A) ,
\end{align*}
where the first map is induced by the open immersion $j : \mathbb{P}^1 \setminus F^\ast \{1\} \hookrightarrow \mathbb{P}^1 \setminus \{0\}$.
Now, we set
\[
V_3 := F_\ast j^\ast V_2 \in z^r_w (\mathscr{X} \otimes \mathbb{A}^1 \otimes (\mathbb{P}^1 , -\{1\}),q,A).
\]
Consider the automorphism $\tau : \mathbb{P}^1 \to \mathbb{P}^1$, which sends $(0,1,\infty )$ to $(\infty ,0,1)$.
Then, we obtain a cycle \[\widetilde{V}:=\tau^\ast V_3 \in z^r_w (\mathscr{X} \otimes \mathbb{A}^1 \otimes (\mathbb{P}^1 , -\{\infty \}),q,A) = z^r_w (\mathscr{X} \otimes \mathbb{A}^1 \otimes \overline{\square}^{(-1)},q,A).\]
By linearly extending the above construction, and by taking maximum of the integers $m$ for all irreducible components $V$ of $Z$, we obtain a cycle 
\[
\widetilde{Z} \in z^r_w (\mathscr{X} \otimes \mathbb{A}^1 \otimes \overline{\square}^{(-1)},q,A) .
\]
Regard $\widetilde{Z}$ as a cycle in $z^r (\mathscr{X} \otimes \mathbb{A}^1 ,q+1,A)$ by the map $\Phi_q $ constructed in the proof of Lemma \ref{rigidity}.
If $Z = \mathrm{pr}_1^\ast Z' = Z' \times \mathbb{A}^1$ for some $Z'\in z^r (\mathscr{X},q-1,A)$, we can check that 
$\widetilde{Z}=\deg{f} \cdot Z \times \mathbb{A}^1$, noting that $\mu^\ast Z = Z \times \mathbb{A}^1$ and $F_\ast (Z \times \mathbb{A}^1 ) = \deg f \cdot Z \times \mathbb{A}^1$.
Therefore, the class $[\widetilde{Z}]$ in $\frac{z^r_w (\mathscr{X} \otimes \mathbb{A}^1 ,q+1,A)}{\mathrm{pr}_1^\ast z^r (\mathscr{X},q+1,A)}$ is determined by the class $[Z]$ in $\frac{z^r_w (\mathscr{X} \otimes \mathbb{A}^1 ,q,A)}{\mathrm{pr}_1^\ast z^r (\mathscr{X},q,A)}$.
Consider the inclusion morphisms
\[i_\epsilon : \mathscr{X}^\circ \times \mathbb{A}^1 \times \mathbb{A}^q \cong \mathscr{X}^\circ \otimes \mathbb{A}^1 \times \mathbb{A}^q  \otimes \{\epsilon \} \hookrightarrow  \mathscr{X}^\circ \otimes \mathbb{A}^1 \times \mathbb{A}^q  \times \mathbb{P}^1 , \ \ \epsilon \in \{0,1,\infty \} .\]
Noting that the above construction is compatible with the restrictions to the faces of $\mathbb{A}^q$,
if we assume that $Z$ satisfies $d_q [Z] = 0$ in $\frac{z^r_w (\mathscr{X} \otimes \mathbb{A}^1 ,q-1,A)}{\mathrm{pr}_1^\ast z^r (\mathscr{X},q-1,A)}$, then we have $d_{q+1} [\widetilde{Z}] =\pm ( i_0^\ast [\widetilde{Z}] - i_1^\ast [\widetilde{Z}])$ in $\frac{z^r_w (\mathscr{X} \otimes \mathbb{A}^1 ,q,A)}{\mathrm{pr}_1^\ast z^r (\mathscr{X},q,A)}$.
Recall that what we need to prove is that $\mathrm{div}(f) \cdot [Z] = 0$ in $\mathrm{NCH}^r (\mathscr{X},q,A) \cong \mathrm{H}_q \left( \frac{z^r_w (\mathscr{X} \otimes \mathbb{A}^1 ,\ast ,A)}{\mathrm{pr}_1^\ast z^r (\mathscr{X},\ast ,A)} \right)$.
Therefore, we are reduced to showing
\begin{claim} For any $Z \in z^r_w (\mathscr{X} \otimes \mathbb{A}^1 ,q,A)$,  we have
$\mathrm{(1)}$ 
$i_0^\ast \widetilde{Z} = \deg f \cdot \mathrm{pr}_1^\ast (i_0^\ast Z)$, and 
$\mathrm{(2)}$ 
$i_1^\ast \widetilde{Z}=\mathrm{div}(f) \cdot Z$.
\end{claim}
Clearly, we may assume that $Z=V$ is integral.
The assertion (1) follows from the calculation:
\begin{align*}
i_0^\ast \widetilde{V} &= i_\infty^\ast F_\ast j^\ast  V_2
=  \deg f \cdot i_\infty^\ast  j^\ast V_2 = \deg f \cdot i_\infty^\ast  V_2 =  \deg f \cdot i_0^\ast V_1 \\
&= \deg f \cdot i_0^\ast (\mathrm{id}_{\mathscr{X}^\circ } \times \mu )^\ast V = \deg f \cdot \mathrm{pr}_1^\ast (i_0^\ast V).
\end{align*}
 
In the following, we prove (2).
First, we note that $i_1^\ast \widetilde{V} = i_0^\ast F_\ast j^\ast  V_2$.
Consider a cartesian diagram
\[\begin{xymatrix}{
\mathrm{div}(f) \ar[r]^(0.4){i'_0} \ar[d]_{\pi } & \mathbb{P}^1 \setminus F^{-1} (1) \ar[d]^{F} \\
\{0\} \ar[r]_(0.45){i_0 } & \mathbb{P}^1 \setminus \{1\}
}\end{xymatrix}\]
where $\pi$ denotes the structure morphism.
By projection formula, we obtain 
\[i_0^\ast F_\ast j^\ast  V_2 = \pi_\ast i^{\prime \ast }_0 j^\ast V_2 = \pi_\ast i^{\prime \ast }_0 V_2 = \pi_\ast i^{\prime \ast }_0 \lambda^\ast (\mathrm{id}_{\mathscr{X^\circ }} \times \mu)^\ast V.\]
So, it suffices to prove that $\mathrm{div}(f) \cdot V \stackrel{\text{by def.}}{=} \mu_\ast (V \times \mathrm{div}(f)) = \pi_\ast i^{\prime \ast }_0 \lambda^\ast (\mathrm{id}_{\mathscr{X^\circ }} \times \mu)^\ast V$.
By linearity, we are reduced to proving the following lemma.
\begin{lem}
Let $y \in \mathbb{A}^1 \setminus \{0\}$ be a closed point, and set $y^{-1} := \lambda^{-1} y$.
Let $i: y^{-1} \to \mathbb{A}^1$ be the inclusion and let $\pi  : \mathrm{Spec}(k(y)) = \mathrm{Spec}(k(y^{-1}))\to \mathrm{Spec}(k)$ be the structure morphism.
Then, we have \[\mu_\ast (V \times y)=\pi_\ast i^\ast (\mathrm{id}_{\mathscr{X^\circ }} \times \mu)^\ast V.\]
\end{lem}

\begin{proof}
Let $\mathrm{mult}_{y}$ and $\mathrm{mult}_{y^{-1}}$ denote the automorphisms of $\mathbb{A}^1_{k(y)}$ obtained by multiplying $y$ and $y^{-1}$, respectively.
Then, the left hand side can be rewritten as $\pi_\ast (\mathrm{mult}_y )_\ast (\pi^\ast V )$.
On the other hand, the right hand side can be rewritten as $\pi_\ast  (\mathrm{mult}_{y^{-1}} )^\ast (\pi^\ast V)$.
Since $\mathrm{mult}_{y}$ and $\mathrm{mult}_{y^{-1}}$ are inverse to each other, we have $(\mathrm{mult}_y )_\ast = (\mathrm{mult}_{y^{-1}} )^\ast$. This finishes the proof.
 \end{proof}
Thus, we finished the proof of Proposition \ref{vanishing}.
 \end{proof}

Now, we are ready to prove the theorem.
Let $f \in \mathbb{W}(k) = 1+uk\llbracket u \rrbracket$ be a formal power series and let $V \in \mathrm{NCH}^r (\mathscr{X},q,A)$ be an element.
By the lemma, there exists a positive integer $m$ such that $\mathrm{div} (1+u^{m+1} k[u]) \cdot V = 0$ in $\mathrm{NCH}^r (\mathscr{X},q,A)$.
Let $f_m $ denote the image of $f$ under the quotient map $\mathbb{W}(k) \to \mathbb{W}_{m+1} (k) = (1+uk[u])/(1+u^{m+1} k[u])$.
Then, we define 
\[
f \cdot V := \mathrm{div}(f_m ) \cdot V,
\]
where the right hand side is well-defined since $\mathrm{div} (1+u^{m+1} k[u]) \cdot V = 0$.
Moreover, it is independent of the choice of $m$; let $m'$ be another integer with $\mathrm{div} (1+u^{m'+1} k[u]) \cdot V = 0$.
Then, we have $f_m \cdot V = f_{m'} \cdot V$ because $f_m$ and $f_{m'}$ have a common lift in $1+uk[u]$.
Since the association $(f,V) \mapsto f\cdot V$ is obviously bilinear, we obtain a homomorphism of abelian groups
\[
\mathbb{W}(k) \otimes \mathrm{NCH}^r (\mathscr{X},q,A) \to \mathrm{NCH}^r (\mathscr{X},q,A).
\]
Again by lemma, this action is continuous; for any element $V \in \mathrm{NCH}^r (\mathscr{X},q,A)$, take an integer $m$ such that $\mathrm{div}(1+u^{m+1} k[u]) \cdot V=0$.
Then, by definition of the action of $\mathbb{W}(k)$ on $\mathrm{NCH}^r (\mathscr{X},q,A)$, we have $(1+u^{m+1} \cdot k\llbracket u \rrbracket) \cdot V = 1\cdot V= 0$.
In fact, this action gives a {\it module structure}. 
Indeed, the unit acts as the identity since $(1-t) \cdot V = \mathrm{div}(1-t) \cdot V = \{1\} \cdot V = V$.
To check associativity, note that any element of $1+uk\llbracket u \rrbracket$ is uniquely written as a product $\prod_{n \geq 1} (1-a_n u^n )$ with $a_n \in k$.
By continuity of the action, it suffices to prove that $((1-au^m ) \star (1-bu^n )) \cdot V = (1-au^m ) \cdot ((1-b u^n ) \cdot V)$, where $\star$ denotes the product of the big Witt ring $\mathbb{W}(k)$.
If one of $a,b$ is zero, then the both sides of the equation are zero, hence the assertion is trivial. So, we assume that $a \neq 0$ and $b\neq 0$.
Setting $d = \mathrm{g.c.d}(m,n)$, and let $\alpha$ and $\beta $ be solutions of the equations $t^m = a^{-1}$ and $t^n = b^{-1}$, respectively.
Moreover, let $\zeta_l$ denote the $l$-th root of unity for any integer $l > 0$.
Then, we have
\begin{align*}
&\hspace{4.5mm}\mathrm{div}((1-au^m ) \star (1-bu^n )) \\
&= \mathrm{div}((1-a^{n/d} b^{m/d} t^{mn/d})^d ) 
=d \cdot \mathrm{div} (1-a^{n/d} b^{m/d} t^{mn/d}) \\
&=d \cdot \sum_{l=0}^{mn/d -1} \zeta_{mn/d}^l \alpha \beta 
=\mu_\ast \left ( \sum_{i=0}^{m-1} \zeta_m^i \alpha , \sum_{j=0}^{n-1} \zeta_n^j \beta \right) \\
&=\mu_\ast (\mathrm{div}(1-at^m ) , \mathrm{div}(1-bt^n )).
\end{align*}
This proves the associativity, and finishes the proof of Theorem \ref{Witt-action}.

\subsection{A vanishing theorem and homotopy invariance}

The module structure defined in the previous subsection implies the following vanishing theorem.

\begin{thm}\label{vanishing-of-NK}
Let $k$ be a field, and let $A$ be an abelian group.
\begin{itemize}
\item[$\mathrm{(i)}$]
Assume that $\mathrm{char}(k)=p>0$.
Then, for any modulus pair $\mathscr{X}$ over $k$ and for any integers $r$ and $q$, the nilpotent higher Chow group $\mathrm{NCH}^r (\mathscr{X},q,A)$ is a $p$-group.
In particular, if we assume moreover that $A$ is a $\mathbb{Z}[1/p]$-module, then $\mathrm{NCH}^r (\mathscr{X},q,A)=0$.
\item[$\mathrm{(ii)}$]
Assume that $\mathrm{char}(k)=0$.
Then, for any modulus pair $\mathscr{X}$ over $k$ and for any integers $r$ and $q$, the nilpotent higher Chow group $\mathrm{NCH}^r (\mathscr{X},q,A)$ is a $k$-vector space.
In particular, if we assume moreover that $A$ is a torsion abelian group, then $\mathrm{NCH}^r (\mathscr{X},q,A)=0$.
\end{itemize}
\end{thm}

\begin{proof}
Assume that $\mathrm{char}(k)=p>0$, and take any $V \in \mathrm{NCH}^r (\mathscr{X},q,A)$. It suffices to prove that the submodule $\mathbb{W}(k) \cdot V \subset \mathrm{NCH}^r (\mathscr{X},q,A)$ is $p$-primary torsion.
By continuity of the action, the submodule $\mathbb{W}(k) \cdot V$ is in fact a $\mathbb{W}_{p^m } (k)$-module for some $m\geq 0$.
Recalling that the unity of $\mathbb{W}_n (k)$ is the polynomial $1-t$ and that 
\[
(1-t)^{p^{m+1} } = 1 - t^{p^{m+1}} \in 1+t^{p^m +1 } k[t] ,
\]
we can see that $p^{m+1} \mathbb{W}_{p^m } (k) = 0$.
This implies that $p^{m+1} \mathbb{W}(k) \cdot V = 0$, proving $\mathrm{(i)}$.
To prove $\mathrm{(ii)}$, assume that $\mathrm{char}(k)=0$.
Then, $\mathbb{W}(k)$ is isomorphic to the product ring $\prod_{n\geq 1} k$ through the ghost operator.
This finishes the proof.
 \end{proof}

\begin{cor}\label{hom-inv-part} 
Assume one of the following conditions:
\begin{itemize}
\item[$\mathrm{(i)}$]
$\mathrm{char}(k)=p>0$ and $A$ is a $\mathbb{Z}[1/p]$-module, or
\item[$\mathrm{(ii)}$]
$\mathrm{char}(k)=0$ and $A$ is a torsion abelian group.
\end{itemize}
Then, 
for any modulus pair $\mathscr{X}$ over $k$, the pullback by the first projection induces an isomorphism
\[
\mathrm{CH}^r (\mathscr{X},q,A) \xrightarrow{\cong } \mathrm{CH}^r (\mathscr{X}\otimes \mathbb{A}^1 ,q,A).
\]
\end{cor}

\begin{proof}
Obvious by the above theorem and Proposition \ref{nilp-split} .
 \end{proof}

\begin{rem}\label{direct-proof}
We can prove the assertion (i) of Corollary \ref{hom-inv-part} by directly using Bloch's technique which was used in the proof of $\mathbb{A}^1$-homotopy invariance of the higher Chow group as in \cite[Theorem 2.1]{B} (see also Section \ref{pf-of-inv}).
Let us give a brief sketch below.
Since we have the moving lemma \ref{special-moving-lemma}, it suffices to prove the rigidity lemma, which states that the pullback maps $i_0^\ast , i_1^\ast : z^r_w (\mathscr{X}\otimes \mathbb{A}^1 ,\ast ,A) \to z^r (\mathscr{X},\ast ,A)$ are homotopic, where $w=\{\mathscr{X}^\circ \times \{0,1\}\}$.
By Lemma \ref{colimit-lemma}, we have  $z^r_w (\mathscr{X}\otimes \mathbb{A}^1 ,q ,A) = \bigcup_{m \geq 1} z^r_w (\mathscr{X}\otimes \overline{\square}^{(-p^m )} ,q ,A)$ for each $q$.
So, we are reduced to proving that the maps $i_0^\ast , i_1^\ast : z^r_w (\mathscr{X}\otimes  \overline{\square}^{(-p^m )} ,\ast ,A) \to z^r (\mathscr{X},\ast ,A)$ are homotopic for any $m$.
Consider the finite morphism $\rho : \mathbb{A}^1 \to \mathbb{A}^1$ which sends the parameter $t$ to $t^{p^m}$. 
Then, we have a commutative diagram
\[\begin{xymatrix}{
z^r_w (\mathscr{X} \otimes \overline{\square}^{(-m)} ,\ast ,A) \ar[rr]^{(\mathrm{id}_{\mathscr{X}^\circ } \times \rho )_\ast } \ar[rd]_{p^m \cdot i_\epsilon^\ast } && z^r_w (\mathscr{X} \otimes \overline{\square}^{(-1)} ,\ast ,A) \ar[ld]^{i_\epsilon^\ast } \\
& z^r (\mathscr{X} ,\ast ,A) &
}\end{xymatrix}\]
for each $\epsilon = 0,1$.
The horizontal map is well-defined since $\rho^\ast \{\infty \} = p^m \{\infty \}$.
The commutativity is trivial when $\epsilon = 0$.
When $\epsilon = 1$, it follows from the equation $(t-1)^{p^m } = t^{p^m }-1$.
Then, by Lemma \ref{rigidity}, we can see that $p^m i_0^\ast $ and $p^m i_1^\ast$ are homotopic.
Since $A$ is a $\mathbb{Z}[1/p]$-module by assumption, the multiplication by $p^m$ is invertible on $z^r (\mathscr{X},\ast ,A)$.
This proves the desired assertion.
\end{rem}

\section{Independence theorem}\label{section-indep}


In this section, we obtain Theorem \ref{main-4} from the following result:

\begin{thm}\label{comparison-general}
Let $k$ be a field and let $A$ be an abelian group.
Assume one of the following conditions:
\begin{itemize}
\item[$\mathrm{(i)}$]
$\mathrm{char}(k)=p>0$ and $A$ is a $\mathbb{Z}[1/p]$-module, or
\item[$\mathrm{(ii)}$]
$\mathrm{char}(k)=0$ and $A$ is a torsion abelian group.
\end{itemize}
Then, for any modulus pair $\mathscr{X}$ over $k$ and for any integers $r,q$, there exists a canonical isomorphism
\[
\mathrm{CH}^r (\mathscr{X},q ,A)  \cong \mathrm{CH}^r (\mathscr{X} ,q,A)' ,
\]
where the right hand side denotes the na\"{i}ve higher Chow group with modulus (see Definition \ref{variant-higher-chow-mod}).
\end{thm}

If we consider effective modulus pairs, we obtain the following corollary by Remark \ref{effective-case}.

\begin{cor}\label{comparison-reduced}
($\Longrightarrow$ Theorem \ref{main-4})
Let $k,\mathscr{X},A$ be as in Theorem \ref{comparison-general}.
Assume moreover that $\mathscr{X}=(\overline{X},X^\infty )$ is effective.
Then, for any $r,q \geq 0$, 
the higher Chow group $\mathrm{CH}^r (\mathscr{X},q,A)$ is independent of the multiplicity of the Cartier divisor $X^\infty$.
In other words, if $X^{\prime \infty }$ is another effective Cartier divisor on $\overline{X}$ such that $|X^{\prime \infty }| = |X^\infty |$ and if we set $\mathscr{X}':=(\overline{X},X^{\prime \infty })$, then we have a canonical isomorphism 
\[
\mathrm{CH}^r (\mathscr{X},q,A)
\cong
\mathrm{CH}^r (\mathscr{X}' ,q,A).
\]
\end{cor}


\begin{proof}[Proof of Theorem \ref{comparison-general}]
We imitate Bloch's argument in \cite[Theorem 4.3]{B2}, which was used in the proof of the comparison isomorphism between cubical higher Chow group and simplicial higher Chow group.
Define a an abelian group $\mathcal{Z}(a,b)$ for any $a,b \in \mathbb{Z}$ by
\[
\mathcal{Z}(a,b) := \mathbb{Z} \Set{\parbox{85mm}{irreducible closed subsets $V \subset \mathscr{X}^\circ \times \mathbb{A}^{a+b}$ of codimension $r$ meeting $\mathscr{X}^\circ \times F$ properly for any face $F \subset \mathbb{A}^{a+b}$ and satisfying the modulus condition $(\star \star )$}} \otimes_{\mathbb{Z}} A,
\]
where the modulus condition is as follows:
\begin{itemize}
\item[$(\star \star )$]  Let $\overline{V} \subset \overline{X} \times (\mathbb{P}^1 )^{a} \times \mathbb{A}^b$ be the closure of $V$ and let $\overline{V}^N \to \overline{V}$ be the normalization.
Then, the following inequality of Cartier divisors on $\overline{V}^N$ holds:
\[
X^\infty \times (\mathbb{P}^1 )^{a} \times \mathbb{A}^b \leq \overline{X} \times F_a \times \mathbb{A}^b .
\]
\end{itemize}
By the same argument as Proposition \ref{chow-cubical}, we obtain a bifunctor
\[
\mathcal{Z}(\underline{\ast },\underline{\ast }) : \mathbf{Cube}^\mathrm{op} \times \mathbf{Cube}^\mathrm{op} \to 
\mathbf{Ab}.
\]
Let $\mathcal{Z}(\ast ,\ast )$ be the associated double complex.
Note that the differentials 
\begin{align*}
d' : \mathcal{Z}(a,\ast  ) \to \mathcal{Z}(a-1 ,\ast  ), \ \ 
d'' : \mathcal{Z}(\ast  ,b) \to \mathcal{Z}(\ast ,b-1),
\end{align*}
which are defined as the appropriate alternating sums of the maps induced by the pullback of cycles along the inclusions $\mathbb{A}^{a-1} \hookrightarrow \mathbb{A}^{a} $ and $\mathbb{A}^{b-1} \hookrightarrow \mathbb{A}^{b}$ (which are well-defined by the containment lemma) also form maps of cubical groups
$d' : \mathcal{Z}(a,\underline{\ast } ) \to \mathcal{Z}(a-1 ,\underline{\ast })$ and
$d'' : \mathcal{Z}(\underline{\ast } ,b) \to \mathcal{Z}(\underline{\ast },b-1)$.
We also define double subcomplexes
\[
\mathcal{Z}(\ast ,\ast )_l, \mathcal{Z}(\ast ,\ast )_r \subset \mathcal{Z}(\ast, \ast )
\]
by 
\begin{align*}
\mathcal{Z}(a,b)_l &:= \Set{V \in \mathcal{Z}(a,b) | \parbox{70mm}{$V$ vanishes on $\{x_i = 0\}$ for each $i=1,\dots ,a$, where $(x_1 ,\dots ,x_a )$ is the coordinate on $\mathbb{A}^a $}}, \\
\mathcal{Z}(a,b)_r &:= \Set{V \in \mathcal{Z}(a,b) | \parbox{70mm}{$V$ vanishes on $\{y_j = 0\}$ for each $j=1,\dots ,b$, where $(y_1 ,\dots ,y_b )$ is the coordinate on $\mathbb{A}^b $}}.
\end{align*}
Furthermore, we set
\[
\mathcal{Z}(\ast ,\ast )_0 := \mathcal{Z}(\ast ,\ast )_l \cap \mathcal{Z}(\ast ,\ast )_r .
\]
Consider the following two kinds of (homological) convergent spectral sequences ${}'E,{}''E$ associated to the double complex $\mathcal{Z}(\ast ,\ast )_0$ :
\begin{align*}
{}'E^2_{b,a} &= {}''H_b ({}'H_a (\mathcal{Z}(\ast ,\ast )_0 )) \Longrightarrow H_{a+b} (\mathrm{Tot}(\mathcal{Z}(\ast ,\ast )_0 )) \\
{}''E^2_{a,b} &= {}'H_a ({}''H_b (\mathcal{Z}(\ast ,\ast )_0 )) \Longrightarrow H_{a+b} (\mathrm{Tot}(\mathcal{Z}(\ast ,\ast )_0 ))
\end{align*}
where ${}'H,{}''H$ are the homology groups with respect to the differentials $d',d''$, respectively.
$\mathrm{Tot}$ denotes the functor which sends a double complex to the associated total complex of it.
Consider the following claim:
\begin{claim}\label{degeneration}
\begin{align*}
{}'E^1_{b,a}&=
\begin{cases}
	0 & (b>0) \\
	\mathrm{CH}^r (\mathscr{X},a,A) & (b=0)
\end{cases}  \cdots (1) \\
{''}E^1_{a,b}&=
\begin{cases}
	0 & (a>0) \\
	\mathrm{CH}^r (\mathscr{X} ,b,A)' & (a=0)
\end{cases} \cdots (2)
\end{align*}
\end{claim}

\begin{rem}
Assuming Claim \ref{degeneration}, we obtain the desired isomorphism
\[
\mathrm{CH}^r (\mathscr{X},n,A) \cong H_n (\mathrm{Tot}(\mathcal{Z}(\ast ,\ast )_0 )) \cong \mathrm{CH}^r (\mathscr{X} ,n,A)' .
\]
\end{rem}

In the following, we prove Claim \ref{degeneration}.
The assertion (1) follows from the $\mathbb{A}^1 $-homotopy invariance of $\mathrm{CH}^r (-,\ast ,A)$ proved in Theorem \ref{hom-inv-part}.
Indeed, for any $b \geq 0$, we have
\[
\mathcal{Z}(\ast ,b)_l = \underline{z}^r_w (\mathscr{X} \otimes \mathbb{A}^b ,\ast ,A)_0 ,
\]
where $w = \Set{\mathscr{X}^\circ \times (\text{faces of }\mathbb{A}^b )}$.
So, we have
\begin{align*}
H_a (\mathcal{Z}(\ast ,b)_l ) 
&= H_a (\underline{z}^r_w (\mathscr{X} \otimes \mathbb{A}^b ,\ast ,A)_0 )
\cong H_a (\underline{z}^r (\mathscr{X} \otimes \mathbb{A}^b ,\ast ,A)_0 ) \\
&= \mathrm{CH}^r (\mathscr{X} \otimes \mathbb{A}^b ,a,A) \cong \mathrm{CH}^r (\mathscr{X},a,A) \\
&= H_a (\underline{z}^r (\mathscr{X},\ast ,A)_0 )   =  H_a (\mathcal{Z}(\ast ,0)_l ),
\end{align*}
where the first isomorphism follows from Theorem \ref{special-moving-lemma} and Remark \ref{remark-normalized}.  
The second one follows from Theorem \ref{hom-inv-part}.

Thus, we can see that the pullback of cycles by the inclusion $\mathscr{X}^\circ \cong \mathscr{X}^\circ \times \{(0,\dots ,0)\} \hookrightarrow \mathscr{X}^\circ \times \mathbb{A}^b $ induces a quasi-isomorphism of complexes
\[
\mathcal{Z}(\ast ,b )_l \xrightarrow{\simeq } \mathcal{Z}(\ast ,0)_l = \underline{z}^r (\mathscr{X},\ast ,A )_0 .
\]
Note that, if $b>0$, the composite
$
\mathcal{Z}(\ast ,b)_0 \subset \mathcal{Z}(\ast ,b)_l \xrightarrow{\simeq } \mathcal{Z}(\ast ,0)_l
$
is zero by definition of $\mathcal{Z}(\ast ,b)_0$.

\begin{lem}\label{splitting}
$\mathcal{Z}(\ast ,b)_0$ is a direct summand of $\mathcal{Z}(\ast ,b)_l $ as a complex.
\end{lem}

\begin{proof}
To see this, for any $a \geq 0$, we regard $\mathcal{Z}(a,\ast )_l = \mathcal{Z}(a,\ast )_l $ as a cubical group.
Then, $\mathcal{Z}(a,\ast )_0 = (\mathcal{Z}(a,\ast  )_l )_0$ (see Lemma \ref{normalization} for the notation of the right hand side).
Since the differential $d' : \mathcal{Z}(a,\ast  )_l \to \mathcal{Z}(a-1,\ast  )_l $ is a morphism of cubical groups, 
it is compatible with the splitting of complexes in Lemma \ref{normalization} (2).
This finishes the proof.
 \end{proof}
Thus, we have ${}' H_a (\mathcal{Z}(\ast ,b)_0 )=0$ for any $a \geq 0$ and for any $b>0$.
If $b=0$, then $\mathcal{Z}(\ast ,0)_0 = \mathcal{Z}(\ast ,0)_l = z^r (\mathscr{X},\ast ,A)_0 $, hence the assertion (1) follows.

\ 

We prove the assertion (2) in the claim.
For any $a \geq 0$, we have
\[
\mathcal{Z}(a,\ast )_r = \underline{z}^r_w (\mathscr{X} \otimes \overline{\square}^{(-a)}  ,\ast ,A)'_0 ,
\]
where $w=\Set{\mathscr{X}^\circ \times (\text{faces of } \mathbb{A}^a )}$.
Thus, we have
\begin{align*}
H_b (\mathcal{Z}(a,\ast )_r ) 
&= H_b (\underline{z}^r_w (\mathscr{X} \otimes \overline{\square}^{(-a)} ,\ast ,A)'_0 ) \cong H_b (\underline{z}^r (\mathscr{X} \otimes \overline{\square}^{(-a)} ,\ast ,A)'_0 ) \\
&= \mathrm{CH}^r (\mathscr{X} \otimes \overline{\square}^{(-a)} ,b,A)'  \cong \mathrm{CH}^r (\mathscr{X},b,A)' \\
&= H_b (\underline{z}^r (\mathscr{X},\ast ,A)'_0 ) = H_b (\mathcal{Z}(\ast ,0)_r ),
\end{align*}
where the first isomorphism follows from Theorem \ref{special-moving-lemma} and Remark \ref{remark-normalized}.  
The second one follows from Theorem \ref{hom-inv-part}.

Note that, if $a>0$, the composite
\[
\mathcal{Z}(a ,\ast )_0 \subset \mathcal{Z}(a ,\ast )_r \stackrel{\simeq }{\to} \mathcal{Z}(a ,\ast )_r
\]
is zero by definition of $\mathcal{Z}(a,\ast )_0$.
By the same argument as Lemma \ref{splitting}, we can prove that $\mathcal{Z}(a,\ast )_0$ is a direct summand of $\mathcal{Z}(a,\ast )_r$.
Thus, we have \[{}' H_b (\mathcal{Z}(a,\ast )_0 )=0\] for any $b \geq 0$ and for any $a>0$.
If $a=0$, then $\mathcal{Z}(0,\ast )_0 = \mathcal{Z}(0,\ast )_r = z^r (\mathscr{X},\ast ,A)'_0 $, hence the assertion (2) follows.
This finishes the proof of Claim \ref{degeneration}, hence that of Theorem \ref{comparison-general}.
 \end{proof}

As a corollary, we have the following result, which describes the ``difference" of $\mathrm{CH}^r (\overline{X}|X^\infty  ,q)$ and $\mathrm{CH}^r (\overline{X}|X^{\prime \infty },q)$ for effective divisors $X^\infty \leq X^{\prime \infty }$ with the same support.

\begin{cor}\label{difference-str}
Let $k$ be a field and let $A$ be an abelian group.
Let $\mathscr{X}=(\overline{X},X^\infty)$ and $\mathscr{X}'=(\overline{X},X^{\prime \infty })$ be two effective modulus pairs over $k$ with $|X^\infty|=|X^{\prime \infty }|$ and $X^{\infty } \leq X^{\prime \infty }$.
Consider the quotient complex $\frac{z^r (\mathscr{X} ,\ast ,A)}{z^r (\mathscr{X}',\ast ,A)}$.
Then, for any $q$, the homology group $H_q \left( \frac{z^r (\mathscr{X} ,\ast ,A)}{z^r (\mathscr{X}',\ast ,A)}\right)$ is:
\begin{itemize}
\item[$\mathrm{(1)}$]
a $p$-group if $\mathrm{char}(k)=p>0$, and
\item[$\mathrm{(2)}$] 
uniquely divisible i.e. a $\mathbb{Q}$-vector space if $\mathrm{char}(k)=0$.
\end{itemize}
\end{cor}

\begin{proof}
The case (1) immediately follows from Corollary \ref{comparison-reduced} since \[H_q \left( \frac{z^r (\mathscr{X} ,\ast ,A)}{z^r (\mathscr{X}',\ast ,A)}\right) \otimes_\mathbb{Z} \mathbb{Z}[1/p] = H_q \left( \frac{z^r (\mathscr{X} ,\ast ,A\otimes_\mathbb{Z} \mathbb{Z}[1/p])}{z^r (\mathscr{X}',\ast ,A\otimes_\mathbb{Z} \mathbb{Z}[1/p])}\right)=0.\]
Next we prove the case (2).
For any abelian group $B$, define
\[
M_\ast (\mathscr{X}',\mathscr{X},B):=\frac{z^r (\mathscr{X} ,\ast , B)}{z^r (\mathscr{X}',\ast ,B)}.
\]
Then, we have the following lemma.
\begin{lem}
Let $0 \to A_1 \to A_2 \to A_3 \to 0$ be a short exact sequence of abelian groups.
Then, the induced sequence of complexes
\[
0 \to M_\ast (\mathscr{X}',\mathscr{X},A_1 ) \to M_\ast (\mathscr{X}',\mathscr{X},A_2 ) \to M_\ast (\mathscr{X}',\mathscr{X},A_3 ) \to 0
\]
is also exact.
\end{lem}
\begin{proof}
Apply the snake lemma to the injective map between short exact sequences
\[
\begin{xymatrix}{
0 \ar[r] & z^r (\mathscr{X}' ,\ast , A_1 ) \ar[r] \ar@{^{(}->}[d] & z^r (\mathscr{X}' ,\ast , A_2 ) \ar[r] \ar@{^{(}->}[d] & z^r (\mathscr{X}' ,\ast , A_3 ) \ar[r] \ar@{^{(}->}[d] & 0 \\
0 \ar[r] & z^r (\mathscr{X} ,\ast , A_1 ) \ar[r] & z^r (\mathscr{X} ,\ast , A_2 ) \ar[r] & z^r (\mathscr{X} ,\ast , A_3 ) \ar[r] & 0
}\end{xymatrix}\]
 \end{proof}
Let $A$ be an abelian group, and let $A_\mathrm{tors}$ denote the maximal torsion subgroup of $A$.
Then, by the above lemma, we have a short exact sequence
\[
0 \to M_\ast (\mathscr{X}',\mathscr{X},A_\mathrm{tors} ) \to M_\ast (\mathscr{X}',\mathscr{X},A ) \xrightarrow{\alpha } M_\ast (\mathscr{X}',\mathscr{X},A/A_\mathrm{tors} ) \to 0.
\]
Since $M_\ast (\mathscr{X}',\mathscr{X},A_\mathrm{tors} )$ is quasi-isomorphic to $0$ by Corollary \ref{comparison-reduced}, the map $\alpha $ is a quasi-isomorphism.
Therefore, we may assume that $A$ is a torsion-free abelian group.

We assume that $A$ is torsion-free, and set $A_\mathbb{Q} := A \otimes_\mathbb{Z} \mathbb{Q}$ and $A_{\mathbb{Q}/\mathbb{Z}} := A \otimes_\mathbb{Z} \mathbb{Q}/\mathbb{Z}$.
Since $A$ is torsion-free, we have a short exact sequence
\[
0 \to A \to A_\mathbb{Q} \to A_{\mathbb{Q}/\mathbb{Z}} \to 0.
\]
Then, again by the above lemma, we obtain a short exact sequence 
\[
0 \to M_\ast (\mathscr{X}',\mathscr{X},A ) \to M_\ast (\mathscr{X}',\mathscr{X},A_\mathbb{Q} ) \to M_\ast (\mathscr{X}',\mathscr{X},A_{\mathbb{Q}/\mathbb{Z}} ) \to 0.
\]
Since $A_{\mathbb{Q}/\mathbb{Z}}$ is a torsion abelian group, the complex $M_\ast (\mathscr{X}',\mathscr{X},A_{\mathbb{Q}/\mathbb{Z}})$ is quasi-isomorphic to $0$.
Therefore, for any $q \geq 0$, we have  \[\mathrm{H}_q (M_\ast (\mathscr{X}',\mathscr{X},A)) \cong \mathrm{H}_q (M_\ast (\mathscr{X}',\mathscr{X},A_{\mathbb{Q}})) \cong \mathrm{H}_q (M_\ast (\mathscr{X}',\mathscr{X},A)) \otimes_\mathbb{Z} \mathbb{Q},\]
where the right hand side is uniquely divisible.
This finishes the proof of the corollary.
 \end{proof}

\subsection*{A result on relative motivic cohomologies}

For any modulus pair $\mathscr{X}=(\overline{X},X^\infty )$ and for any $r,q \geq 0$, we have an  \'{e}tale sheaf $z^r ((-)_{\mathscr{X}},q)$ on $\overline{X}$ defined by 
\[
(\text{\'{e}tale schemes over } \overline{X}) \ni U \mapsto z^r (U_{\mathscr{X}} ,q) \ \ (U_{\mathscr{X}} :=(U,X^\infty |_U )).
\]
For any abelian group $A$,
we also have an \'{e}tale sheaf 
\[
z^r ((-)_{\mathscr{X}},q,A)=z^r ((-)_{\mathscr{X}},q) \otimes_\mathbb{Z} A
\]
since $z^r (U_{\mathscr{X}} ,q)$ is a free $\mathbb{Z}$-module for each $q$.
Then, we obtain a cohomological complex of sheaves on the Nisnevich site $\overline{X}_{\mathrm{Nis}}$:
\[
A(r)_{\mathscr{X}} := z^r ((-)_{\mathscr{X}} ,2r-\ast ,A).
\]

\begin{dfn}\label{relative-mot}
For a modulus pair $\mathscr{X}$ over $k$ and for any $r,q \geq 0$, we define the {\it relative motivic cohomology groups of $\mathscr{X}$ with coefficients in $A$} as the hypercohomology group
\[
H_{\mathcal{M}}^q (\mathscr{X},A(r))
:=
\mathbb{H}^q_{\mathrm{Nis}} (\overline{X},A(r)_{\mathscr{X}} ).
\]
\end{dfn}

\begin{rem}
When $\mathscr{X}$ is effective, the relative motivic cohomology is introduced in \cite{BS}. If moreover $\mathscr{X}^\circ$ is smooth, the relative motivic cohomology is contravariantly functorial with respect to coadmissible left fine (not necessarily flat) morphisms of modulus pairs (\cite{Kai}).
We do not know whether this is true for non-effective modulus pairs or not.
\end{rem}

\begin{thm}\label{comparison-reduced-coh} 
Let $k$ be a field and let $A$ be an abelian group.
Assume one of the following conditions:
\begin{itemize}
\item[$\mathrm{(i)}$]
$\mathrm{char}(k)=p>0$ and $A$ is a $\mathbb{Z}[1/p]$-module, or
\item[$\mathrm{(ii)}$]
$\mathrm{char}(k)=0$ and $A$ is a torsion abelian group.
\end{itemize}
Then, for any effective modulus pair $\mathscr{X}=(\overline{X},X^\infty )$ over $k$ and for any $r,q \geq 0$, the relative motivic cohomology group $H_{\mathcal{M}}^q (\mathscr{X},A(r))$  
 is independent of the multiplicity of the Cartier divisor $X^\infty$.
In other words, if $X^{\prime \infty }$ is another effective Cartier divisor on $\overline{X}$ such that $|X^{\prime \infty }| = |X^\infty |$ and we set $\mathscr{X}':=(\overline{X},X^{\prime \infty })$, then we have a canonical isomorphism 
\[
H_{\mathcal{M}}^q (\mathscr{X},A(r)) \cong H_{\mathcal{M}}^q (\mathscr{X}',A(r)).
\]
\end{thm}

\begin{proof}
By Theorem \ref{comparison-general}, Remark \ref{effective-case}, the diagram of complexes of presheaves
\[
z^r (-_\mathscr{X} ,\ast ,A) \hookrightarrow z^r (-_\mathscr{X} ,\ast ,A)' = z^r (-_{\mathscr{X}'} ,\ast ,A)' \hookleftarrow z^r (-_{\mathscr{X}'} ,\ast ,A)
\]
induces a canonical isomorphism in the derived category of complexes of sheaves $z^r (-_\mathscr{X} ,\ast ,A) \cong z^r (-_{\mathscr{X}'} ,\ast ,A)$.
This finishes the proof.
 \end{proof}

\addcontentsline{toc}{section}{References}

\end{document}